\documentclass[12pt,a4paper,reqno]{amsart}
\usepackage{amsthm}
\usepackage{amsmath}
\usepackage{color}
\usepackage{graphicx}
\usepackage{amsfonts}
\usepackage{xfrac}
\usepackage{mathrsfs}
\numberwithin{equation}{section}
\begin{document}
\title{Logarithm laws for the BCZ map}
\author{Y.~Li}
\begin{abstract}
We present the logarithm laws for the partial sum of the itinerary function over non-periodic BCZ orbits, utilizing the even and odd Diophantine exponent defined by Athreya-Margulis \cite{At1}. We also give a detailed description of the BCZ map and its excursions.
\end{abstract}
\maketitle

\newtheorem{ddd}{Definition}[section]
\newtheorem{ttt}[ddd]{Theorem}
\newtheorem{rrr}[ddd]{Remark}
\newtheorem{ppp}[ddd]{Proposition}
\newtheorem{ccc}[ddd]{Corollary}
\newtheorem{llll}[ddd]{Lemma}
\newtheorem{con}[ddd]{Conjecture}
\renewcommand{\bold}[1]{\smallskip \noindent {\bf \boldmath #1 }\nopagebreak[4]}

\section{Introduction}\label{sec:intro}
\raggedbottom
\allowdisplaybreaks
Logarithm laws for diagonalizable flows were first studied by Sullivan \cite{Su}, who showed the logarithm laws for the cusp excursion of geodesic flow. Later, Masur \cite{Ma} extended the result for geodesic flow in the moduli space. Then, Kleinbock-Margulis \cite{Kl} studied the actions of one-parameter diagonalizable subgroups on non-compact finite-volume homogeneous spaces in a more general context. Subsequently, Athreya-Margulis \cite{At1} extended the previous result to the context of unipotent flows and provided further results on logarithm laws for horocycle flows in \cite{At2}.

\bold{BCZ map.}
This paper presents the logarithm laws for the BCZ map regarding its itinerary function. Boca-Cobeli-Zaharescu~\cite{Bo} introduced the BCZ map when proving a conjecture of R.R.Hall \cite{Ha} about the asymptotic analysis of the sum of the squares of the h-gap differences of the Farey sequence. Motivated by the study of orbits of the horocycle flow, which was previously unknown to the BCZ map, Athreya-Cheung \cite{At} showed that the BCZ map can describe the first return map on the Poincar\'{e} section
$$\Omega':=\{\Lambda_{a,b}:=p_{a,b}SL(2,\mathbb{Z})\in\mathbb{R}^2| a,b\in(0,1], a+b>1\}$$
where $$p_{a,b}=
\begin{pmatrix}
a & b\\
0 & a^{-1}
\end{pmatrix}
$$
for the horocycle flow
$$ h_s=
\begin{pmatrix}
1 & 0\\
-s & 1
\end{pmatrix}
:s\in\mathbb{R}$$
on the space of unimodular lattices $X_2=SL(2,\mathbb{R})/SL(2,\mathbb{Z})$ in $\mathbb{R}^2$(\cite{At}, Theorem 1.1). Let $$\Omega:=\left\{(a,b)\in\mathbb{R}^2 \mid a,b\in(0,1],a+b>1\right\}\subset\mathbb{R}^2$$
be the Farey triangle. The first return map $T:\Omega\to\Omega$ defined implicitly by
\begin{equation}\Lambda_{T(a,b)}=h_{R(a,b)}\Lambda_{a,b}, \label{34} \end{equation}
where the first return time $R(a,b)=\frac{1}{ab}$
is given explicitly by the BCZ map
$$T(a,b)=\left(b,-a+\left[\frac{1+a}{b}\right] b\right).$$
We call $k(a,b):=\left[\frac{1+a}{b}\right]$ the itinerary function (or index function) of $(a,b)$.

Our main result can be considered as an extension of Athreya-Margulis's result (\cite{At1}, Proposition 3.3) mainly due to the relationship between horocycle flow and BCZ map. Also, although studying different objects, they surprisingly share the same terms, which we will explain in \S1.3.

\bold{Orbit and itinerary.}
For the Farey sequence of order $n$:
$\rho_0=0<\rho_1=\frac{1}{n}<\rho_2=\frac{1}{n-1}<\rho_3<\rho_4<\cdots<\rho_{A_n-1}=\frac{n-1}{n}<\rho_{A_n}=1$, we extend the sequence by setting $\rho_i=\rho_{A_n+i}$ for all $i\in\mathbb{Z}$ and let $\rho_i=\frac{p_i}{q_i}$ where $p_i,q_i\in\mathbb{N}$ and $(p_i,q_i)=1$.

Boca-Cobeli-Zaharescu \cite{Bo} showed that
$$T\left(\frac{q_{k}}{n},\frac{q_{k+1}}{n}\right)=\left(\frac{q_{k+1}}{n},\frac{q_{k+2}}{n}\right)$$
holds for all $k\in\mathbb{Z}$.
Therefore, $T^k\left(\frac{1}{n},1\right)=\left(\frac{q_{k}}{n},\frac{q_{k+1}}{n}\right)$, which implies $\left(\frac{1}{n},1\right)$'s orbit is periodic with period $A_n$.

In fact, for any rational point $(a,b)\in\Omega$, which means $\frac{b}{a}$ is a rational number, its orbit is periodic. Let $\frac{b}{a}=\frac{p}{q}$, $n=\left[\frac{q}{a}\right]$, then the period of the orbit is $A_n$.
For any $(a,b)\in\Omega$, we denote by $(a_n,b_n)=T^n(a,b)$ for $n\in\mathbb{Z}$ the orbit of $(a,b)$ under the BCZ map, namely the BCZ orbit. We call $(a_n,b_n)_{n=1}^{+\infty}$ the forward orbit, and $(a_n,b_n)_{n=-\infty}^0$ the backward orbit. Then, we define the itinerary sequence for $(a,b)$ as $k_n(a,b):=k\left(T^{n-1}(a,b)\right)=\left[\frac{b_n+1}{a_n}\right]$ for $n\in\mathbb{Z}$. Base on the definition of BCZ map, $a_{n+1}=b_n=k(a_{n-1},b_{n-1})b_{n-1}-a_{n-1}=k(a_{n-1},b_{n-1})a_n-a_{n-1}$. Therefore,
\begin{equation}k_n(a,b)=\frac{a_{n-1}+a_{n+1}}{a_n}   \label{484848}  \end{equation}
When $n>0$, we refer to it as the forward itinerary; whereas for $n\leq0$, we call it the backward itinerary.

The research on the itinerary function of the BCZ map has been extensive. Hall-Shiu \cite{sh}, Hall \cite{Ha1} gave some properties about the itinerary function. Boca-Gologan-Zaharescu \cite{Go} presented some asymptotic formulas regarding the distribution of the itinerary function for periodic BCZ orbits. Alkan-Ledoan-Zaharescu \cite{Al1} and Alkan-Ledoan-V\^aj\^aitu-Zaharescu \cite{Al2} showed some results about the moments of the itinerary function for periodic BCZ orbits.

\bold{Relationship with the RH.}
Zagier \cite{Za} showed that proving an optimal rate of equidistribution for long periodic trajectories of the horocycle flow on $X_2$ (that is, an optimal error term in Sarnak's theorem \cite{Sa}) is equivalent to the classical Riemann hypothesis (RH).

In a forthcoming work, a characterization of the RH given by Franel \cite{Fr} and Landau \cite{La} is reinterpreted in terms of estimates of $L^1$-averages of the BCZ cocycle along periodic orbits of the BCZ map. Subsequently, we consider the problem where the cocycle is replaced by a discrete approximation using $\hat{k}-3:=\frac{k+k^T}{2}-3$, where $k^T(a,b)=k\left(T^{-1}(a,b)\right)=k(b,a)$. We establish that the discretized analog of the RH holds in a stronger sense. This exciting discovery suggests a possible new approach to the RH.

\bold{Partial sum of the itinerary.}
Hall-Shiu \cite{sh} showed that over a periodic orbit, we have
$$\sum_{i=1}^{A_n}k_i\left(\frac{1}{n},1\right)=3A_n-1.$$
Furthermore, for any rational point $(a,b)$ within the periodic orbit with period $A_n$, we also have $\displaystyle\sum_{i=1}^{A_n}k_i\left(a,b\right)=3A_n-1$. Notice that the average value of $k$ in $\Omega$ is $3$, so subtracting  $3A_n$ from both sides yields $\displaystyle\sum_{i=1}^{A_n}\left(k_i\left(a,b\right)-3\right)=-1$. Additionally, Boca-Gologan-Zaharescu \cite{Go} demonstrated that for every positive integer $h$, there exists a rational number $A(h)$ such that
$$\displaystyle\sum_{i=1}^{A_n}\left(k_i\left(a,b\right)k_{i+h}\left(a,b\right)-A(h)\right)=O_h\left(n\log^2n\right).$$

As for the irrational point $(a,b)$, its orbit is non-periodic. One natural question arises: what can be said about the partial sum of the itinerary function over the irrational point's non-periodic orbit, defined as $f_n=\displaystyle\sum_{i=1}^n\left(k_i\left(a,b\right)-3\right)$?

The main result (Theorem \ref{8817}, \ref{8818}) of this paper is the calculation of the limit superior of the log speed of the partial sum of the itinerary function over the irrational point's non-periodic orbit. In fact, we have
$$\limsup_{n\rightarrow+\infty}\frac{\log|f_n|}{\log n}=\max\left\{\frac{e^--1}{e^-},\frac{e^+-2}{e^+}\right\}$$
where $e^+$,$e^-$ represent the even Diophantine exponent and odd Diophantine exponent, respectively (Definition \ref{494949}), which Athreya-Margulis \cite{At1} defined. Furthermore, our result can also be applied to $\hat{k}-3$ (Corollary \ref{9898}).

\subsection{Plan of paper}
In the remainder of the introduction, we define the even and odd Diophantine exponent using two equivalent descriptions. Then we present the main result of this paper (Theorem \ref{8817}, \ref{8818}). In \S2, we provide the preliminary knowledge of the BCZ map in order to prove the main result. We give a detailed description of the BCZ map using the unimodular lattice, define the excursion of the BCZ map and estimate the length of the excursion. In \S3, it is purely technical and we give the full proof of Theorem \ref{8817}. In \S3.1, we define the $h$ function and demonstrate its property (Lemma \ref{97743}). In \S3.4, we prove the overall monotonicity of $h$ in an excursion (Theorem \ref{773}). In \S3.2, \S3.3, \S3.5, and \S3.6, we discuss the main results under different scenarios, where we estimate the log speed in different sections of orbit. In \S3.7, we summarize the results. In \S3.8, we prove the corollary of Theorem \ref{8817} where the $k$ is replaced by $\hat{k}$.

\subsection{Even and odd Diophantine exponent}
For an irrational number $s$, its Diophantine exponent is
$$e(s):=\sup\left\{v\in\mathbb{R}:\left|s-\frac{p}{q}\right|<\frac{1}{q^v}\text{ has infinitely many solutions }\frac{p}{q}\in\mathbb{Q}\right\}.$$

\begin{ddd}(\cite{At1}, section 3)   \label{494949}
For an irrational number $s$, we define its even Diophantine exponent to be
$$e^+(s):=\sup\left\{v\in\mathbb{R}:0<\frac{p}{q}-s<\frac{1}{q^v}\text{ has infinitely many solutions }\frac{p}{q}\in\mathbb{Q}\right\};$$
its odd Diophantine exponent to be
$$e^-(s):=\sup\left\{v\in\mathbb{R}:0<s-\frac{p}{q}<\frac{1}{q^v}\text{ has infinitely many solutions }\frac{p}{q}\in\mathbb{Q}\right\}.$$
\end{ddd}

Let $[c_0;c_1,c_2,\cdots]$ be the continued fraction of $s$, $\frac{p_k'}{q_k'}=[c_0;c_1,c_2\cdots c_k]$, $(k\geq0)$. Let $c_n={q_{n-1}'}^{e_n-2}$ $(n>0)$. The relation between $e$ and $e_n$ can be described by the following equation:
$$e=\inf\left\{e'|c_{n}=O\left({q_{n-1}'}^{e'-2}\right)\right\}=\limsup_{n\rightarrow\infty}e_n;$$
the relation between $e^+$ and $e_n$ can be described by the following equation:
$$e^+=\inf\left\{e'|c_{2n}=O\left({q_{2n-1}'}^{e'-2}\right)\right\}=\limsup_{n\rightarrow\infty}e_{2n};$$
the relation between $e^-$ and $e_n$ can be described by the following equation:
$$e^-=\inf\left\{e'|c_{2n+1}=O\left({q_{2n}'}^{e'-2}\right)\right\}=\limsup_{n\rightarrow\infty}e_{2n+1}.$$
Furthermore, the relationship among $e$, $e^+$ and $e^-$ is $e=\max\{e^+,e^-\}$. We also know that $e,e^+,e^-\geq2$.

\subsection{Main result}

\begin{ttt}  \label{8817}
For an irrational point $(a,b)$ in the BCZ triangle, let $f_n=\displaystyle\sum_{i=1}^n\left(k_i(a,b)-3\right)$ for $n\geq1$. Given that $e^+$ and $e^-$ are the even and odd Diophantine exponent of $\frac{b}{a}$, respectively, we have

$$\limsup_{n\rightarrow+\infty}\frac{\log|f_n|}{\log n}=\max\left\{\frac{e^--1}{e^-},\frac{e^+-2}{e^+}\right\}.$$

\end{ttt}

\begin{ttt}   \label{8818}
For an irrational point $(a,b)$ in the BCZ triangle, let $f_{-n}=\displaystyle\sum_{i=0}^n(k_{-i}(a,b)-3)$ for $n\geq1$. Given that $e^+$ and $e^-$ are the even and odd Diophantine exponent of $\frac{b}{a}$, respectively, we have

$$\limsup_{n\rightarrow+\infty}\frac{\log|f_{-n}|}{\log n}=\max\left\{\frac{e^+-1}{e^+},\frac{e^--2}{e^-}\right\}.$$

\end{ttt}

The asymmetry of these results is partly due to the fact that the log speeds of the horocycle flow's cusp excursion from opposite directions are different. Athreya-Margulis \cite{At1} showed a result stating that given $s\in\mathbb{R}$, let
$$\Lambda_s:=\begin{pmatrix}
1&s\\
0&1
\end{pmatrix}
\mathbb{Z}^2.
$$
Define $\alpha_1: X_2\rightarrow\mathbb{R}^+$ by
$$
\alpha_1(\Lambda):=\sup_{0\neq v\in\Lambda}\frac{1}{\Vert v\Vert}$$
where $\Vert\cdot\Vert$ denotes Euclidean norm on $\mathbb{R}^2$.

\begin{ppp}(\cite{At1}, Proposition 3.3)If $s\notin\mathbb{Q}$, $e^+=e^+(s)$, $e^-=e^-(s)$, then
$$\limsup_{t\rightarrow+\infty}\frac{\log(\alpha_1(h_t\Lambda_s))}{\log(|t|)}=\frac{e^--1}{e^-},$$
$$\limsup_{t\rightarrow-\infty}\frac{\log(\alpha_1(h_t\Lambda_s))}{\log(|t|)}=\frac{e^+-1}{e^+}.$$
\end{ppp}

We can see that this result shares the same terms with Theorem \ref{8817}, and \ref{8818}. The extra terms in Theorem \ref{8817}, and \ref{8818} come from the BCZ map's unique structure, which also make the main result much harder to prove than the propositions above although they appear similar.

Since almost all irrational numbers' Diophantine exponents are $2$, which means $e^+=e^-=2$. Thus for almost all $s$, we have
$$\limsup_{|t|\rightarrow\infty}\frac{\log(\alpha_1(h_t\Lambda_s))}{\log(|t|)}=\frac{1}{2},$$
which can also be deduced by Theorem 2.1 of \cite{At1}.

Similarly, we can have the following corollary from Theorem \ref{8817}, and \ref{8818}.

\begin{ccc}  \label{8832}
For almost every irrational point $(a,b)$ in the BCZ triangle, the Diophantine exponent of $\frac{b}{a}$ is $2$, so both its even Diophantine exponent and odd Diophantine exponent are $2$, which means that
$$\limsup_{n\rightarrow+\infty}\frac{\log |f_n|}{\log n}=\limsup_{n\rightarrow+\infty}\frac{\log |f_{-n}|}{\log n}=\frac{1}{2}.$$
\end{ccc}

We mentioned previously that $\hat{k}-3$ can be of significance in future research, and we have discovered that the main result can still apply to it.

\begin{ccc} \label{9898}
Theorem \ref{8817}, \ref{8818}, and Corollary \ref{8832} still hold true if we replace $k_i-3=k\circ T^{i-1}-3$ with $\hat{k}\circ T^{i-1}-3$.

\end{ccc}

\section{Preliminary}

$\Omega'$ comprises all unimodular lattices that have a primitive point lying in $(0,1]\times\{0\}$. In the case of $\Lambda_{a,b}$, this primitive point is $(a,0)$. Under the horocycle flow $h_t$, every point of $\Lambda_{a,b}$ in $(0,1]\times\mathbb{R}^+$ maintains its x-coordinate while its slope decreases by $t$, resulting in a downward trajectory along a straight line. This implies that under $h_t$, each primitive point of $\Lambda_{a,b}$ in $(0,1]\times\mathbb{R}^+$ will successively intersect $(0,1]\times\{0\}$ at the time of its slope's value. Specifically, when $(b,\frac{1}{a})$, which is the primitive point in $(0,1]\times\mathbb{R}^+$ with the smallest slope, intersects in $(0,1]\times\{0\}$ and becomes $(b,0)$ at time $t=R(a,b)=\frac{1}{ab}$, the lattice transforms into $\Lambda_{T(a,b)}$. And that's the meaning of (\ref{34}).

For a BCZ orbit $T^i(a_0,b_0)=(a_i,b_i)$, if we select the primitive point of $\Lambda_{a_0,b_0}$ in $(0,1]\times\mathbb{R}^+$ with the smallest slope, then previously mentioned, its x-coordinate is $a_1=b_0$, which is the x-coordinate of $T(a_0,b_0)=(a_1,b_1)$. If we continue this process iteratively, we can establish the following:
\begin{rrr}   \label{81}
The x-coordinate of the primitive point of $\Lambda_{a_0,b_0}$ within $(0,1]\times\mathbb{R}^+$ having the $i$-th smallest slope is $a_i$, which is the x-coordinate of $T^i(a_0,b_0)=(a_i,b_i)$ for $i\geq1$.
\end{rrr}

\begin{ddd}   \label{74}   (excursion)

Let the depth of a point $(a,b)$ be $\frac{1}{a}$. If two points of a BCZ orbit are deeper than all the points lying between these two points, we refer to the process from one of these two points to the other as an excursion. For example, let $(a_n,b_n)=T^n(a,b)$, $s<t$, $s,t\in\mathbb{Z}$. If
$$\frac{1}{a_m}<\frac{1}{a_s},\frac{1}{a_t},$$
for any $m\in(s,t)$, then the process that starts from $(a_s,b_s)$ and stops when it reaches $(a_t,b_t)$ for the first time is considered an excursion. In the sequel, for convenience, we will omit the phrase ``for the first time" and simply refer to the process from $(a_s,b_s)$ to $(a_t,b_t)$ as an excursion.

\end{ddd}

For an excursion from $(a_0,b_0)$ to $T^s(a_0,b_0)=(a_s,b_s)$, let $O$ be $(0,0)$, $A$ be $(a_0,0)$, and $S$ be the primitive point of $\Lambda_{a_0,b_0}$ in $(0,1]\times\mathbb{R}^+$ with the s-th smallest slope. By Remark \ref{81}, we know that $a_s$ is the x-coordinate of $S$, and $a_0$ is the x-coordinate of $A$. Additionally, for $i\in[1,s-1]$, $a_i$ are the x-coordinates of all the primitive points of $(0,1]\times\mathbb{R}^+$ whose slopes are between 0 and $S$'s slope.

Then according to the definition of the excursion, we know that $a_i>a_0,a_s$ for $i\in(0,s)$, which means that there are no other points inside the triangle $OAS$ (including the sides). Thus, the area of $OAS$ is $\frac{1}{2}$ according to Pick's Theorem, considering that $\Lambda_{a_0,b_0}$ is a unimodular lattice. This implies that $\overrightarrow{OA}$ and $\overrightarrow{OS}$ form a set of basis of $\Lambda_{a,b}$. Therefore, the coordinate of $S$ is $(a_s,\frac{1}{a_0})$, and all the primitive points of $(0,1]\times\mathbb{R}^+$ whose slopes are between 0 and $S$'s slope can be generated by $(a_0,0)$ and $(a_s,\frac{1}{a_0})$.

\begin{rrr}   \label{64}
From above, we know that for the primitive point with the $i$-th smallest slope, where $i\in[1,s-1]$, there exists a coprime pair $(u_i,v_i)$ such that the coordinate of this primitive point is given by
$$u_i\left(a_0,0\right)+v_i\left(a_s,\frac{1}{a_0}\right)=\left(u_ia_0+v_ia_s,\frac{v_i}{a_0}\right).$$
Since $a_i$ is the x-coordinate of this primitive point, we have
$$a_i=u_ia_0+v_ia_s.$$
So for every $a_i$, there exists a coprime pair $(u_i,v_i)$ such that $u_ia_0+v_ia_s=a_i\leq1$. Conversely, for every coprime pair $(u,v)$ such that $ua_0+va_s\leq1$, $u(a_0,0)+v(a_s,\frac{1}{a_0})$ represents one of the primitive points in $(0,1]\times\mathbb{R}^+$ whose slopes are between 0 and $S$'s slope. This means there exists an $a_i$ for $i\in[1,s-1]$ such that $ua_0+va_s=a_i$.

In summary, we have the following one-to-one correspondence:
$$\{a_i\mid i\in[1,s-1]\}\longleftrightarrow\{(u,v)\mid u,v\in\mathbb{Z}^+,(u,v)=1,ua_0+va_s\leq1\}.$$

The number of these primitive points is $s-1$ which is also the number of pairs of coprime positive integers $(u,v)$ such that $ua_0+va_s\leq1$.

For $1\leq i<j\leq s-1$, since the slope of $(u_ia_0+v_ia_s,\frac{v_i}{a_0})$ is smaller than the slope of $(u_ja_0+v_ja_s,\frac{v_j}{a_0})$, we have
$$\frac{v_i}{u_i}<\frac{v_j}{u_j}.$$
Two points worth mentioning are $a_1$ and $a_{s-1}$.
For $a_1$, we know that $\frac{v_1}{u_1}<\frac{v_i}{u_i}$ for $i\in[2,s-1]$, since $1\leq v_i$, $\left[\frac{1-a_s}{a_0}\right]\geq u_i$, $\left[\frac{1-a_s}{a_0}\right]a_0+a_s\leq1$, we can deduce that $u_1=\left[\frac{1-a_s}{a_0}\right]$, $v_1=1$, $a_1=\left[\frac{1-a_s}{a_0}\right]a_0+a_s$.
Similarly, we have $u_{s-1}=1$, $v_{s-1}=\left[\frac{1-a_0}{a_s}\right]$, $a_{s-1}=a_0+\left[\frac{1-a_0}{a_s}\right]a_s$.

\end{rrr}

Next, we give the estimate of the length of the excursion.

\begin{llll}       \label{65}     (length of the excursion)

For an excursion $(a_i,b_i)_{i=0}^s$ where $s>0$, $a_0=a$, $a_s=b$, we have
$$s=\frac{3}{\pi^2}\frac{1}{ab}+O\left(\max\left\{\frac{1}{a},\frac{1}{b}\right\}\log\left(\min\left\{\frac{1}{a},\frac{1}{b}\right\}\right)\right).$$

\end{llll}

\begin{proof}[proof]

By Remark \ref{64}, we know that
\begin{equation}
s-1=\mid\{(u,v)\mid ua+vb\leq1,(u,v)=1\}\mid.        \label{45}
\end{equation}

Without the loss of generality, we assume that $a\geq b$.

\begin{align*}
s-1&=\sum_{(u,v)=1, \text{ }ua+vb\leq1}1\\
&=\sum_{ua+vb\leq1, \text{ }d\mid u,\text{ }d\mid v}\mu(d)\\
&=\sum_{d\leq\frac{1}{a}}\mu(d)\sum_{ua+vb\leq\frac{1}{d}}1\\
&=\sum_{d\leq\frac{1}{a}}\mu(d)\left(\frac{1}{2abd^2}+O\left(\frac{1}{bd}\right)\right)\\
&=\frac{1}{2ab}\sum_{d\leq\frac{1}{a}}\frac{\mu(d)}{d^2}+O\left(\frac{1}{b}\log\left(\frac{1}{a}\right)\right)\\
&=\frac{1}{2ab}\sum_{d=1}^\infty\frac{\mu(d)}{d^2}+O\left(\frac{1}{b}\log\left(\frac{1}{a}\right)\right)\\
&=\frac{3}{\pi^2}\frac{1}{ab}+O\left(\frac{1}{b}\log\left(\frac{1}{a}\right)\right).
\end{align*}

\end{proof}

Theorem \ref{65} is equivalent to Theorem 3.1 of \cite{No}.

$\Lambda_{a,b}$ is a lattice generated by $(a,0)$ and $(b,\frac{1}{a})$. We sort
$$\left\{\left(qb-pa,\frac{q}{a}\right)|\text{ primitive point of }\Lambda_{a,b}, qb-pa\in(0,1]\right\}$$
in ascending order based on the slope of each point measured from the origin. Let the sequence be $(q_nb-p_na,\frac{q_n}{a}),n\in\mathbb{Z}$ such that $q_1=1,p_1=0, q_0=0, p_0=-1$. According to the given conditions, we have $p_n,q_n>0$ for $n>1$ and $p_n,q_n<0$ for $n<0$. Furthermore, we have $\frac{p_i}{q_i}<\frac{p_j}{q_j}<\frac{b}{a}$ for $i<j$.

Let$(a_n,b_n)=T^n(a,b)$, then based on Remark \ref{81}, we have
$$a_n=q_nb-p_na,$$
$$b_n=\frac{q_n}{a}.$$

For any irrational point $(a,b)$ in the BCZ triangle, consider the subsequence $(a_{n_i},b_{n_i})_{i=0}^{+\infty}$ of the forward orbit of $(a,b)$ ($0<n_i<n_j$ for $i<j$). This subsequence is chosen such that for each $i$, $a_{n_i}<a_k$ for all $0\leq k< n_i$. In other words, every point in the subsequence is deeper than all previous points in the forward orbit. We want to know the explicit expression of $(a_{n_i},b_{n_i})$.

\begin{llll}   \label{882}
$(a,b)$ is an irrational point in the BCZ triangle. Let $[c_0;c_1,c_2,\cdots]$ be the continued fraction of $\frac{b}{a}$, $\frac{p_k'}{q_k'}=[c_0;c_1,c_2\cdots c_k]$, $(k\geq0)$. Then
$$p_{n_i}=mp_{2k+1}'+p_{2k}',$$
$$q_{n_i}=mq_{2k+1}'+q_{2k}',$$
$$ (a_{n_i},b_{n_i})=(q_{n_i}b-p_{n_i}a,q_{n_i+1}b-p_{n_i+1}a),$$
where $m=i-c_2-c_4-\cdots-c_{2k}$ for $i\in(c_2+c_4+\cdots +c_{2k},c_2+c_4+\cdots +c_{2k+2}]$, $k>0$; $m=i$ for $i\in[0,c_2]$.

\end{llll}

\begin{proof}[proof]
For a positive primitive pair $(\tilde{q},\tilde{p})$ such that $\frac{\tilde{p}}{\tilde{q}}<\frac{b}{a}$, the fact that it corresponds to some $(q_{n_i},p_{n_i})$ is equivalent to for any $\frac{p}{q}<\frac{\tilde{p}}{\tilde{q}}$, we have $\tilde{q}b-\tilde{p}a<qb-pa$, which in other words, means there is no primitive point in the triangle delimited by $y=0$, $x=\tilde{q}b-\tilde{p}a$ and $y=\frac{\frac{\tilde{q}}{a}}{\tilde{q}b-\tilde{p}a}x$ (sides included, vertices excluded). According to the basic property of continued fraction, this is equivalent to $\frac{\tilde{p}}{\tilde{q}}$ being a best approximation of $\frac{b}{a}$. However, we know that all the left-side best approximation of $\frac{b}{a}$ are $\frac{mp_{2k+1}'+p_{2k}'}{mq_{2k+1}'+q_{2k}'}$, $(k\geq0, 0\leq m\leq c_{2k}-1)$, so we have
$$\left\{a_{n_i}|i\geq0\right\}=\left\{\left(mq_{2k+1}'+q_{2k}'\right)b-\left(mp_{2k+1}'+p_{2k}'\right)a|k\geq0, 0\leq m\leq c_{2k}-1\right\}.$$
We sort $(mq_{2k+1}'+q_{2k}')b-(mp_{2k+1}'+p_{2k}')a$ by the order of size from large to small, which is also the order of size of $mq_{2k+1}'+q_{2k}'$ from small to large. Considering that $a_{n_i}$ is already sorted by the order of size from large to small, we can determine the matching between $a_{n_i}$ and $(mq_{2k+1}'+q_{2k}')b-(mp_{2k+1}'+p_{2k}')a$, which is the statement we give in the lemma.

\end{proof}

Since $q_{2k}'b-p_{2k}'a\rightarrow0$ as $k\rightarrow+\infty$, it follows that $a_{n_i}\rightarrow0$ as well. This explains why the orbit of the irrational point is non-periodic and unbounded. Let $a_k'=|q_k'b-p_k'a|$, we know that $a_k'=q_k'b-p_k'a>0$ when $k$ is even, and $a_k'=-(q_k'b-p_k'a)>0$ when $k$ is odd. Furthermore, define $c_k'=c_2+c_4+\cdots+c_{2k}$ and let $n_{c_k'}=n_k'$ for notational convenience. Then according to the notation in Lemma \ref{882}, we have $q_{2k}'=q_{n_k'}$ and $p_{2k}'=p_{n_k'}$.\\

Next, we provide estimations for $n_i-n_{k-1}'$ and $n_k'$.

Before proceeding, it is necessary to clarify some notation. $a_n\sim b_n$ indicates that $\displaystyle\lim_{n\rightarrow\infty}\frac{a_n}{b_n}=1$; $a_n\asymp b_n$ implies that there exists $0<A<B$ such that $A<\frac{a_n}{b_n}<B$.

\begin{llll}   \label{9932}
For $i\in(c_{k-1}',c_k']$, let $s=i-c_{k-1}'$, we have
$$n_i-n_{k-1}'\sim \frac{3}{\pi^2}\frac{1}{a_{2k-1}'}\left(\frac{1}{a_{2k-2}'-sa_{2k-1}'}-\frac{1}{a_{2k-2}'}\right), $$
$$n_k'\asymp \frac{1}{a_{2k-1}'a_{2k}'}.$$

\end{llll}

\begin{proof}
In the first step, we give an estimate of $n_i-n_{k-1}'$.

By Lemma \ref{882}, we have
$$a_{n_{c_{k-1}'+s}}=a_{2k-2}'-sa_{2k-1}'$$
for $s\in[0,c_{2k}]$. 

Since it is an excursion from $(a_{n_{m-1}},b_{n_{m-1}})$ to $(a_{n_m},b_{n_m})$, by Lemma \ref{65}, we have
$$n_m-n_{m-1}\sim \frac{3}{\pi^2}\cdot\frac{1}{a_{n_m}a_{n_{m-1}}},$$
thus
\begin{align}
n_i-n_{k-1}'&= n_i-n_{c_{k-1}'}\nonumber\\   \nonumber
&=\sum_{m=1}^{s}(n_{c_{k-1}'+m}-n_{c_{k-1}'+m-1})\\ \nonumber
&\sim\sum_{m=1}^{s}\frac{3}{\pi^2}\cdot\frac{1}{a_{n_{c_{k-1}'+m}}a_{n_{c_{k-1}'+m-1}}}\\ \nonumber
&=\frac{3}{\pi^2}\sum_{m=1}^{s}\frac{1}{(a_{2k-2}'-ma_{2k-1}')(a_{2k-2}'-(m-1)a_{2k-1}')}\\ \nonumber
&=\frac{3}{\pi^2}\sum_{m=1}^{s}\frac{1}{a_{2k-1}'}\left(\frac{1}{a_{2k-2}'-ma_{2k-1}'}-\frac{1}{a_{2k-2}'-(m-1)a_{2k-1}'}\right)\\
&=\frac{3}{\pi^2}\frac{1}{a_{2k-1}'}\left(\frac{1}{a_{2k-2}'-sa_{2k-1}'}-\frac{1}{a_{2k-2}'}\right).            \label{9914}
\end{align}

When $i=c_k'$, $s=c_{2k}$, $a_{2k-2}'-sa_{2k-1}'=a_{2k}'$, we have
$$n_k'-n_{k-1}'\sim \frac{3}{\pi^2}\frac{1}{a_{2k-1}'}\left(\frac{1}{a_{2k}'}-\frac{1}{a_{2k-2}'}\right).$$

In the second step, we estimate $n_{k}'$.

If we select all the points in the backward orbit that are deeper than the subsequent points in the backward orbit, then by symmetry,  we can determine that $a_{2k-1}'=p_{2k-1}a-q_{2k-1}b$  is the x-coordinate of the $c_1+c_3+\cdots+c_{2k-1}$-th point: $M_k$. Let $T^{m_k}(M_k)=(a_{n_k'},b_{n_k'})$. Since $M_k$ and $(a_{n_k'},b_{n_k'})$ are deeper than any other points in between, it is an excursion from $M_k$ to $(a_{n_k'},b_{n_k'})$. Therefore, by Lemma \ref{65}, we have
$$m_k\sim \frac{3}{\pi^2}\frac{1}{a_{2k-1}'a_{2k}'}.$$

Two obvious relations are $n_k'<m_k$ and $n_k'>n_k'-n_{k-1}'$. Since $a_{2k-2}'\geq a_{2k}'+a_{2k-1}'>2a_{2k}'$, by (\ref{9914}), we have

\begin{align*}
\frac{3}{\pi^2}\frac{1}{2a_{2k-1}'a_{2k}'}&<\frac{3}{\pi^2}\frac{1}{a_{2k-1}'}\left(\frac{1}{a_{2k}'}-\frac{1}{a_{2k-2}'}\right)\\
&\sim n_k'-n_{k-1}'\\
&<n_k'\\
&<m_k\\
&\sim \frac{3}{\pi^2}\frac{1}{a_{2k-1}'a_{2k}'},
\end{align*}
thus we conclude that
$$n_k'\asymp \frac{1}{a_{2k-1}'a_{2k}'}.$$\\

\end{proof}

\section{Proof of Theorem \ref{8817}}
In \S3.1, we introduce the $h$ function and its property. In \S3.2, we calculate the log speed at $n=n_{c_k'}$, In \S3.3, we calculate the log speed at $n=n_i,i\neq c_k'$. In \S3.4, we introduce the overall monotonicity of the $h$ function at an excursion. In \S3.5, we calculate the log speed for $n\in(n_{i-1},n_i),i\in[c_k'+1,c_{k+1}'-1]$. In \S3.6, we calculate the log speed for $n\in(n_{c_k'-1},n_{c_k'})$. In \S3.7, we summarize the results. In \S3.8, we prove Corollary \ref{9898}.

\subsection{h function}
We define a function $h$ on the space of positive number sequences with length greater than or equal to $3$:
$$h(s_1,s_2,\cdots,s_n)=\sum_{i=2}^{n-1}(\frac{s_{i-1}+s_{i+1}}{s_i}-3).$$
By (\ref{484848}), we have $f_n=h(a_0,a_1,\cdots,a_{n+1})=h(a_{-1},a_0,a_1,\cdots,a_{n+1})-(k_0-3)$.

\begin{llll}  \label{97743}
For a positive number sequence $(s_1,s_2,\cdots,s_n)$ with $n\geq5$, if $s_i\mid s_{i-1}+s_{i+1}$ for all $i\in[2,n-1]$, and there exists a largest term $s_r$ (i.e., $s_r\geq s_i$ for all $i\in[1,n]$) where $r\in[3,n-2]$, then eliminating $s_r$ from the sequence does not change the value of $h$.

\end{llll}
\begin{proof}
Since $s_r\mid s_{r-1}+s_{r+1}$, we have $\frac{s_{r-1}+s_{r+1}}{s_r}=1$. When $s_r$ is eliminated from the sequence, we have $\frac{s_{r-2}+s_{r+1}}{s_{r-1}}=\frac{s_{r-2}+s_r}{s_{r-1}}-1$ and $\frac{s_{r-1}+s_{r+2}}{s_{r+1}}=\frac{s_r+s_{r+2}}{s_{r+1}}-1$, Therefore,
$$h\left(s_1,s_2,\cdots,s_{r-1},s_{r+1},\cdots,s_n\right)=h\left(s_1,s_2,\cdots,s_n\right).$$\\
\end{proof}

\subsection{$n=n_{c_k'}$}

Firstly, we calculate the log speed at $n_{c_k'}$.

\subsubsection{Estimation of $f_{n_{c_k'}}$}

To estimate $f_{n_{c_k'}}$, we consider the sequence $\left(a_{-1},a_0,a_1,\cdots,a_{n_{c_k'}},a_{n_{c_k'}+1}\right)$ and apply Lemma \ref{97743} to eliminate terms while preserving the value of the function $h$. We begin by choosing following intervals from the sequence: $[a_{n_i},a_{n_{i+1}}]$ for $0\leq i\leq c_k'-1$, and $[a_0,a_{n_0}]$. In each interval, based on the definition of $a_{n_i}$, any term that is not an endpoint greater than both endpoint term. We identify a largest term and eliminate it from the sequence. According to Lemma \ref{97743}, this elimination does not change the value of $h$.

We repeat this process for each interval, eliminating the largest term in each. After eliminating all terms in these intervals, we obtain the reduced sequence $\left\{a_{-1},a_0,a_{n_0},a_{n_1},a_{n_2},\cdots,a_{n_{c_k'}}\right\}$. Based on our analysis, we can conclude that
$$h\left(a_{-1},a_0,a_{n_0},a_{n_1},a_{n_2},\cdots,a_{n_{c_k'}},a_{n_{c_k'}+1}\right)=h\left(a_{-1},a_0,a_1,\cdots,a_{n_{c_k'}+1}\right).$$

Since
$$\frac{a_{-1}+a_{n_1}}{a_0}=\frac{k_0a-b+b-c_0a}{a}=k_0-c_0,$$
$$\frac{a_0+a_{n_1}}{a_{n_0}}=\frac{a+(q_0+q_1)b-(p_0+p_1)a}{q_0b-p_0a}=\frac{a+(1+c_1)b-(c_0+c_0c_1+1)a}{b-c_0a}=1+c_1,$$
and for $i\in[c_j'+1,c_{j+1}'-1]$ with $j\in[1,k-1]$ and $i\in[1,c_1'-1]$,
$$\frac{a_{n_{i-1}}+a_{n_{i+1}}}{a_{n_i}}=2$$
we have the following:
\begin{align*}
\frac{a_{n_{c_j'-1}}+a_{n_{c_j'+1}}}{a_{n_{c_j'}}}&=\frac{(q_{2j}'-q_{2j-1}')b-(p_{2j}'-p_{2j-1}')a+(q_{2j}'+q_{2j+1}')b-(p_{2j}'+p_{2j+1}')a}{q_{2j}'b-p_{2j}'a}\\
&=\frac{(q_{2j+1}'-q_{2j-1}')b-(p_{2j+1}'-p_{2j-1}')a}{q_{2j}'b-p_{2j}'a}+2\\
&=\frac{c_{2j+1}q_{2j}'b-c_{2j+1}p_{2j}'a}{q_{2j}'b-p_{2j}'a}+2\\
&=c_{2j+1}+2
\end{align*}
for $j\in[1,k-1]$. Consequently,
\begin{align*}
f_{n_{c_k'}}=&h(a_{-1},a_0,a_1,\cdots,a_{n_{c_k'}+1})-(k_0-3)\\
=&h\left(a_{-1},a_0,a_{n_0},a_{n_1},a_{n_2},\cdots,a_{n_{c_k'}},a_{n_{c_k'}+1}\right)-(k_0-3)\\
=&\sum_{i=0}^{n_{c_k'}}\left(\frac{a_{i-1}+a_{i+1}}{a_i}-3\right)-(k_0-3)\\
=&\left(\frac{a_{-1}+a_{n_1}}{a_0}-3\right)+\left(\frac{a_0+a_{n_1}}{a_{n_0}}-3\right)+\sum_{j=1}^{k-1}\left(\frac{a_{n_{c_j'-1}}+a_{n_{c_j'+1}}}{a_{n_{c_j'}}}-3\right)\\
&+\left(\frac{a_{n_{c_k'-1}}+a_{n_{c_k'}+1}}{a_{n_{c_k'}}}-3\right)+\sum_{i=1}^{c_1'-1}\left(\frac{a_{n_{i-1}}+a_{n_{i+1}}}{a_{n_i}}-3\right)\\
&+\sum_{j=1}^{k-1}\sum_{i=c_j'+1}^{c_{j+1}'-1}\left(\frac{a_{n_{i-1}}+a_{n_{i+1}}}{a_{n_i}}-3\right)-(k_0-3)\\
=&(k_0-c_0-3)+(c_1-2)+\sum_{j=1}^{k-1}(c_{2j+1}-1)\\
&+\left(\frac{a_{n_{c_k'-1}}+a_{n_{c_k'}+1}}{a_{n_{c_k'}}}-3\right)-\sum_{i=1}^{c_1'-1}1-\sum_{j=1}^{k-1}\sum_{i=c_j'+1}^{c_{j+1}'-1}1-(k_0-3)\\
=&\sum_{j=0}^{k-1}c_{2j+1}-\sum_{j=0}^kc_{2j}-1+\left(\frac{a_{n_{c_k'-1}}+a_{n_{c_k'}+1}}{a_{n_{c_k'}}}-3\right).
\end{align*}

With regard to $\frac{a_{n_{c_k'-1}}+a_{n_{c_k'}+1}}{a_{n_{c_k'}}}$, we first calculate $\frac{a_{n_{c_k'-1}}}{a_{n_{c_k'}}}$ as following

$$\frac{a_{n_{c_k'-1}}}{a_{n_{c_k'}}}=\frac{a_{n_{c_k'-1}}+a_{n_{c_k'+1}}}{a_{n_{c_k'}}}-\frac{a_{n_{c_k'+1}}}{a_{n_{c_k'}}}=c_{2k+1}+2-\frac{a_{n_{c_k'+1}}}{a_{n_{c_k'}}},$$
where $\frac{a_{n_{c_k'+1}}}{a_{n_{c_k'}}}\in(0,1)$.

Next, we estimate the term $\frac{a_{n_{c_k'}+1}}{a_{n_{c_k'}}}$ using the following inequality:
$$\frac{1}{a_{n_{c_k'}}}-1=\frac{1-a_{n_{c_k'}}}{a_{n_{c_k'}}}<\frac{a_{n_{c_k'}+1}}{a_{n_{c_k'}}}\leq\frac{1}{a_{n_{c_k'}}}.$$

By combining these estimates, we can derive bounds for $f_{n_{c_k'}}$:
$$\sum_{j=0}^k(c_{2j+1}-c_{2j})+\frac{1}{a_{n_{c_k'}}}-4<f_{n_{c_k'}}<\sum_{j=0}^k(c_{2j+1}-c_{2j})+\frac{1}{a_{n_{c_k'}}}-2.$$

Then, we need to estimate $\sum_{j=0}^k(c_{2j+1}-c_{2j})+\frac{1}{a_{n_{c_k'}}}$. We begin by expressing $a_{n_{c_k'}}$ as follows:
$$a_{n_{c_k'}}=q_{n_{c_k'}}b-p_{n_{c_k'}}a=q_{2k}'b-p_{2k}'a=q_{2k}'a\left(\frac{b}{a}-\frac{p_{2k}'}{q_{2k}'}\right).$$
Using following known bounds for $\frac{b}{a}-\frac{p_{2k}'}{q_{2k}'}$:
$$\frac{1}{2q_{2k}'q_{2k+1}'}<\frac{1}{q_{2k}'(q_{2k+1}'+q_{2k}')}<\frac{b}{a}-\frac{p_{2k}'}{q_{2k}'}<\frac{1}{q_{2k}'q_{2k+1}'},$$
we have
$$\frac{q_{2k+1}'}{a}<\frac{1}{a_{n_{c_k'}}}<\frac{2q_{2k+1}'}{a}.$$
To estimate $q_{2k+1}'$, we observe that
$$q_{2k+1}'=c_{2k+1}q_{2k}'+q_{2k-1}'\geq c_{2k+1}q_{2k}'+1\geq c_{2k+1}c_{2k}q_{2k-1}'+c_{2k+1}+1$$
$$\geq c_{2k+1}c_{2k}q_{2k-1}'+2\geq\cdots\geq c_{2k+1}c_{2k}\cdots c_1+(2k+1).$$\\

To further simplify our estimation of $f_{n_{c_k'}}$, we need to consider the relative significance of $\sum\limits_{j=0}^k(c_{2j+1}-c_{2j})$ compared to $\frac{1}{a_{n_{c_k'}}}$. We will demonstrate that the sum becomes negligible as $k$ increases.

We begin with the following inequality:
$$\frac{\sum\limits_{j=0}^k(c_{2j+1}-c_{2j})}{1\big/ a_{n_{c_k'}}}<\frac{\sum\limits_{j=1}^k(c_{2j+1}-1)+c_1}{q_{2k+1}'\big/ a}\leq\frac{\sum\limits_{j=1}^k(c_{2j+1}-1)+c_1}{\prod\limits_{j=1}^{2k+1}c_j+(2k+1)}\leq\frac{\sum\limits_{j=1}^k(c_{2j+1}-1)+c_1}{\prod\limits_{j=1}^kc_{2j+1}+(2k+1)}.$$

If $\sum\limits_{j=1}^{+\infty}(c_{2j+1}-1)+c_1=M<+\infty$, then
$$\frac{\sum\limits_{j=0}^k(c_{2j+1}-c_{2j})}{1\big/ a_{n_{c_k'}}}<\frac{M}{2k+1}\rightarrow 0.$$

If $\sum\limits_{j=1}^{+\infty}(c_{2j+1}-1)+c_1=\infty$, then when $\sum\limits_{j=0}^{k}(c_{2j+1}-1)>3$ and $c_{2k+3}>1$,
\begin{align*}
\frac{\sum\limits_{j=1}^{k+1}(c_{2j+1}-1)}{\prod\limits_{j=1}^{k+1}c_{2j+1}}&=\frac{\sum\limits_{j=1}^k(c_{2j+1}-1)-1}{\prod\limits_{j=1}^{k+1}c_{2j+1}}+\frac{1}{\prod\limits_{j=1}^kc_{2j+1}}\\
&\leq\frac{\sum\limits_{j=1}^k(c_{2j+1}-1)-1}{2\prod\limits_{j=1}^kc_{2j+1}}+\frac{1}{\prod\limits_{j=1}^kc_{2j+1}}\\
&=\frac{\sum\limits_{j=1}^k(c_{2j+1}-1)+1}{2\prod\limits_{j=1}^kc_{2j+1}}\\
&<\frac{2}{3}\cdot\frac{\sum\limits_{j=1}^k(c_{2j+1}-1)}{\prod\limits_{j=1}^kc_{2j+1}}.
\end{align*}
Since there are infinite times that $c_{2k+3}>1$, we have
\begin{equation}\frac{\sum\limits_{j=0}^k(c_{2j+1}-c_{2j})}{1\big/ a_{n_{c_k'}}}<\frac{\sum\limits_{j=1}^k(c_{2j+1}-1)}{\prod\limits_{j=1}^kc_{2j+1}}+\frac{c_1}{\prod\limits_{j=1}^kc_{2j+1}}\rightarrow 0.\label{9876}\end{equation}

On the other hand, we have
$$-\frac{\sum\limits_{j=0}^k(c_{2j+1}-c_{2j})}{1\big/ a_{n_{c_k'}}}<\frac{\sum\limits_{j=0}^k(c_{2j}-1)}{q_{2k+1}'\big/ a}\leq\frac{\sum\limits_{j=0}^k(c_{2j}-1)}{\prod\limits_{j=1}^{2k+1}c_j+(2k+1)}\leq\frac{\sum\limits_{j=0}^k(c_{2j}-1)}{\prod\limits_{j=1}^kc_{2j}+(2k+1)}.$$

If $\sum\limits_{j=1}^{+\infty}(c_{2j}-1)=M<+\infty$, then
$$-\frac{\sum\limits_{j=0}^k(c_{2j+1}-c_{2j})}{1\big/ a_{n_{c_k'}}}<\frac{M}{2k+1}\rightarrow 0.$$

If $\sum\limits_{j=1}^{+\infty}(c_{2j}-1)=\infty$, then when $\sum\limits_{j=1}^{k}(c_{2j}-1)>4-c_0$ and $c_{2k+2}>1$
\begin{align*}
\frac{\sum\limits_{j=0}^{k+1}(c_{2j}-1)}{\prod\limits_{j=1}^{k+1}c_{2j}}&=\frac{\sum\limits_{j=1}^k(c_{2j}-1)+c_0-2}{\prod\limits_{j=1}^{k+1}c_{2j}}+\frac{1}{\prod\limits_{j=1}^kc_{2j}}\\
&\leq\frac{\sum\limits_{j=1}^k(c_{2j}-1)+c_0-2}{2\prod\limits_{j=1}^kc_{2j}}+\frac{1}{\prod\limits_{j=1}^kc_{2j}}\\
&=\frac{\sum\limits_{j=1}^k(c_{2j}-1)+c_0}{2\prod\limits_{j=1}^kc_{2j}}\\
&<\frac{2}{3}\cdot\frac{\sum\limits_{j=0}^k(c_{2j}-1)}{\prod\limits_{j=1}^kc_{2j}}.
\end{align*}
Since there are infinite times that $c_{2k+2}>1$, we have
\begin{equation}-\frac{\sum\limits_{j=0}^k(c_{2j+1}-c_{2j})}{1\big/ a_{n_{c_k'}}}<\frac{\sum\limits_{j=0}^k(c_{2j}-1)}{\prod\limits_{j=1}^kc_{2j}}\rightarrow 0.\label{9877}\end{equation}

Combing (\ref{9876}) and (\ref{9877}), we can conclude that
$$\frac{\sum\limits_{j=0}^k(c_{2j+1}-c_{2j})}{1\big/ a_{n_{c_k'}}}\rightarrow0.$$

\begin{rrr}   \label{771}
From the proof above we know that
$$\frac{\sum\limits_{j=0}^k(c_{2j+1}-c_{2j})}{1\big/ a_{2k}'}\rightarrow0.$$
Take $(\tilde{a},\tilde{b})=(b-c_0a,a)$ and apply the result above, since $\tilde{c_j}=c_{j+1}$ and $\tilde{a_{2k}'}=a_{2k+1}'$, we have
$$\frac{\sum\limits_{j=0}^k(c_{2j+2}-c_{2j+1})}{1\big/ a_{2k+1}'}\rightarrow0.$$\\

\end{rrr}

Then if we pick any constant $r_1\in(0,1)$ and $r_2\in(1,+\infty)$, then as $k\rightarrow+\infty$, we have
\begin{equation}
\frac{r_1}{a_{n_{c_k'}}}<f_{n_{c_k'}}<\frac{r_2}{a_{n_{c_k'}}},     \label{661}
\end{equation}

Therefore,
\begin{equation}
f_{n_{c_k'}}\asymp\frac{1}{a_{n_{c_k'}}}\asymp q_{2k+1}'   \label{662}
\end{equation}

\subsubsection{Estimation of $n_{c_k'}$}

Now, in order to estimate $\log_n|f_n|$ for $n=n_{c_k'}$, we only need to give the estimation of $n_{c_k'}$.

Let $a_k'=|q_k'b-p_k'a|$, we know that $a_k'=q_k'b-p_k'a>0$ when $k$ is even; $a_k'=-(q_k'b-p_k'a)>0$ when $k$ is odd; $a_{n_{c_k'}}=a_{2k}'$. By lemma \ref{9932}, we have
$$n_{c_k'}\asymp \frac{1}{a_{2k-1}'a_{2k}'}.$$
As we proved earlier, it is known that
$$\frac{1}{a_k'}\asymp q_{k+1}',$$
which implies that
\begin{equation}
n_{c_k'}\asymp q_{2k}'q_{2k+1}'.             \label{663}
\end{equation}

\subsubsection{$\limsup\limits_{k\rightarrow+\infty}\log_{n_{c_k'}}|f_{n_{c_k'}}|$}

Using (\ref{662}) and (\ref{663}), we obtain
$$\limsup_{k\rightarrow+\infty}\log_{n_{c_k'}}|f_{n_{c_k'}}|=\limsup_{k\rightarrow+\infty}\log_{n_{c_k'}}\frac{1}{a_{n_{c_k'}}}=\limsup_{k\rightarrow+\infty}\log_{q_{2k}'q_{2k+1}'}q_{2k+1}'$$
\begin{equation}
=\limsup_{k\rightarrow+\infty}\log_{(q_{2k+1}')^{\frac{1}{e_{2k+1}-1}}q_{2k+1}'}q_{2k+1}'=\frac{e^--1}{e^-}.    \label{664}
\end{equation}\\

\subsection{$n=n_i,i\neq c_k'$}

Next, we aim to calculate $\limsup\limits_{i\rightarrow+\infty}\log_{n_i}|f_{n_i}|$. Based on what we proved above, we know that when $i=c_k'$, the limit superior is $\frac{e^--1}{e^-}$. But we also need to obtain the result for $i\in(c_k',c_{k+1}')$, so we need to calculate $\limsup\limits_{i\rightarrow+\infty, i\neq c_k'}\log_{n_i}|f_{n_i}|$. The method of calculation is similar to what we did earlier, which is using the $h$ function and Lemma \ref{97743}.

One thing we need to mention is that if there exists a constant $K>0$ such that for any $k>K$, $c_{2k}=1$, $c_k'=c_{k-1}'+1$, this limit superior won't exist. In this case, $e^+=2$.\\

\subsubsection{Estimation of $f_{n_i},i\neq c_k'$}

For $i\in(c_k',c_{k+1}')$, first we need to estimate $f_{n_i}=h(a_0,a_1,\cdots,a_{n_i+1})$.
\begin{align*}
h(a_0,a_1,\cdots,a_{n_i+1})&=h(a_{-1},a_0,a_1,\cdots,a_{n_i+1})-(k_0-3)\\
&=h(a_{-1},a_0,a_{n_0},a_{n_1},a_{n_2},\cdots,a_{n_i},a_{n_i+1})-(k_0-3)\\
&=\sum_{j=0}^kc_{2j+1}-i-1+(\frac{a_{n_{i-1}}+a_{n_i+1}}{a_{n_i}}-3).
\end{align*}
Since
$$\frac{a_{n_{i-1}}}{a_{n_i}}=\frac{a_{n_{i-1}}+a_{n_{i+1}}}{a_{n_i}}-\frac{a_{n_{i+1}}}{a_{n_i}}=2-\frac{a_{n_{i+1}}}{a_{n_i}}\in(1,2),$$
$$\frac{1}{a_{n_i}}-1=\frac{1-a_{n_i}}{a_{n_i}}<\frac{a_{n_i+1}}{a_{n_i}}\leq\frac{1}{a_{n_i}},$$
we have following bound
$$\sum_{j=0}^kc_{2j+1}-i+\frac{1}{a_{n_i}}-4<f_{n_i}<\sum_{j=0}^kc_{2j+1}-i+\frac{1}{a_{n_i}}-2.$$\\
Let $i'=i-c_k'$, we have
$$\sum_{j=0}^k(c_{2j+1}-c_{2j})-i'+\frac{1}{a_{n_i}}-4<f_{n_i}<\sum_{j=0}^k(c_{2j+1}-c_{2j})-i'+\frac{1}{a_{n_i}}-2,$$
According to Remark \ref{771}, we know that
$$\left|\frac{\sum\limits_{j=0}^k(c_{2j+1}-c_{2j})}{1\big/ a_{n_i}}\right|\leq\left|\frac{\sum\limits_{j=0}^k(c_{2j+1}-c_{2j})}{1\big/ a_{n_{c_k'}}}\right|\rightarrow0.$$
Therefore, if we pick any constant $r_3\in(0,1)$ and $r_4\in(1,+\infty)$, as $i,k\rightarrow+\infty$, we obtain
\begin{equation}
\frac{r_3}{a_{n_i}}-i'<f_{n_i}<\frac{r_4}{a_{n_i}}-i'.        \label{666}
\end{equation}
Let $r_3=0.9$, $r_4=1.1$, we have
$$|f_{n_i}|<\max\left\{\frac{r_4}{a_{n_i}},i'\right\}.$$

\subsubsection{Estimation of $n_i,i\neq c_k'$}

Now, we need to give an estimation about $n_i$. By Lemma \ref{9932}, we know that
$$n_i-n_k'\sim \frac{3}{\pi^2}\frac{1}{a_{2k+1}'}\left(\frac{1}{a_{2k}'-i'a_{2k+1}'}-\frac{1}{a_{2k}'}\right)=\frac{3}{\pi^2}\frac{i'}{(a_{2k}'-i'a_{2k+1}')a_{2k}'}, $$
$$n_k'\asymp \frac{1}{a_{2k-1}'a_{2k}'}.$$
Since
\begin{align*}
n_{c_k'}&<m_k\\
&=\#\{(s,t)|s,t\in\mathbb{Z}^*,(s,t)=1,sa_{2k}'+ta_{2k-1}'\leq 1\}+1\\
&<\#\{(s,t)|s,t\in\mathbb{Z}^*,(s,t)=1,sa_{2k}'+t(a_{2k}'-i'a_{2k+1}')\leq 1\}+1\\
&=n_{c_k'+1}-n_{c_k}'\\
&\leq n_i-n_{c_k'},
\end{align*}
therefore,
$$n_i\in(n_i-n_{c_k'},2(n_i-n_{c_k'})),$$
$$n_i\asymp\frac{i'}{(a_{2k}'-i'a_{2k+1}')a_{2k}'}.$$

\subsubsection{$\limsup\limits_{i\rightarrow+\infty, i\neq c_k'}\log_{n_i}|f_{n_i}|$}

Now, we aim to calculate $\limsup\limits_{i\rightarrow+\infty, i\neq c_k'}\log_{n_i}|f_{n_i}|$.
\begin{align}
\limsup\limits_{i\rightarrow+\infty, i\neq c_k'}\log_{n_i}|f_{n_i}|&\leq\limsup\limits_{i\rightarrow+\infty, i\neq c_k'}\log_{n_i}\max\left\{\frac{r_4}{a_{n_i}},i'\right\}   \nonumber\\  \nonumber
&=\max\left\{\limsup\limits_{i\rightarrow+\infty, i\neq c_k'}\log_{n_i}\frac{r_4}{a_{n_i}},\limsup\limits_{i\rightarrow+\infty, i\neq c_k'}\log_{n_i}i'\right\}\\
&=\max\left\{\limsup\limits_{i\rightarrow+\infty, i\neq c_k'}\log_{n_i}\frac{1}{a_{n_i}},\limsup\limits_{i\rightarrow+\infty, i\neq c_k'}\log_{n_i}i'\right\}.    \label{737674}
\end{align}

For $\limsup\limits_{i\rightarrow+\infty, i\neq c_k'}\log_{n_i}\frac{1}{a_{n_i}}$, we have
$$\limsup\limits_{i\rightarrow+\infty, i\neq c_k'}\log_{n_i}\frac{1}{a_{n_i}}=\limsup\limits_{i\rightarrow+\infty, i\neq c_k'}\log_{\frac{i'}{(a_{2k}'-i'a_{2k+1}')a_{2k}'}}\frac{1}{a_{2k}'-i'a_{2k+1}'}.$$
Since $i<c_{k+1}'$, it follows $a_{2k}'-i'a_{2k+1}'>a_{2k+1}'$. From this, we can deduce
$$\frac{2i'^2}{2i'-1}<i'+1<\frac{a_{2k}'}{a_{2k+1}'},$$
which implies
$$\frac{a_{2k}'}{i'}<2(a_{2k}'-i'a_{2k+1}'),$$
and hence
$$\frac{i'}{(a_{2k}'-i'a_{2k+1}')a_{2k}'}>\frac{1}{2(a_{2k}'-i'a_{2k+1}')^2}.$$
Therefore,
$$\limsup\limits_{i\rightarrow+\infty, i\neq c_k'}\log_{n_i}\frac{1}{a_{n_i}}\leq\frac{1}{2}.$$

If we pick $i=c_k'+1$, $i'=1$, then
$$\frac{1}{(a_{2k}'-a_{2k+1}')a_{2k}'}<\frac{1}{(a_{2k}'-a_{2k+1}')^2},$$
so
$$\limsup\limits_{i\rightarrow+\infty, i\neq c_k'}\log_{n_i}\frac{1}{a_{n_i}}\geq\limsup\limits_{i\rightarrow+\infty, i=c_k'+1<c_{k+1}'}\log_{n_i}\frac{1}{a_{n_i}}\geq\frac{1}{2}.$$
Therefore,
\begin{equation}
\limsup\limits_{i\rightarrow+\infty, i\neq c_k'}\log_{n_i}\frac{1}{a_{n_i}}=\frac{1}{2}.   \label{737675}
\end{equation}\\

For $\limsup\limits_{i\rightarrow+\infty, i\neq c_k'}\log_{n_i}i'$, we have
\begin{align}
\limsup\limits_{i\rightarrow+\infty, i\neq c_k'}\log_{n_i}i'&=\limsup\limits_{i\rightarrow+\infty, i\neq c_k'}\log_{\frac{i'}{(a_{2k}'-i'a_{2k+1}')a_{2k}'}}i'  \nonumber\\   \nonumber
&=\frac{1}{\liminf\limits_{i\rightarrow+\infty, i\neq c_k'}\log_{i'}\frac{i'}{(a_{2k}'-i'a_{2k+1}')a_{2k}'}}\\
&=\frac{1}{\liminf\limits_{i\rightarrow+\infty, i\neq c_k'}\log_{i'}\frac{1}{(a_{2k}'-i'a_{2k+1}')a_{2k}'}+1},   \label{687653}
\end{align}
For $i'<c_{2k+2}$, we have $\frac{1}{(a_{2k}'-i'a_{2k+1}')a_{2k}'}>\frac{1}{a_{2k}'^2}$, thus
\begin{align*}
\liminf\limits_{i\rightarrow+\infty, i\neq c_k'}\log_{i'}\frac{1}{(a_{2k}'-i'a_{2k+1}')a_{2k}'}&\geq\liminf\limits_{k\rightarrow+\infty,  c_k'+1<c_{k+1}'}\log_{c_{2k+2}}\frac{1}{a_{2k}'^2}\\
&=\liminf\limits_{k\rightarrow+\infty,  c_k'+1<c_{k+1}'}\log_{c_{2k+2}}q_{2k+1}'^2\\
&=\liminf\limits_{k\rightarrow+\infty,  c_k'+1<c_{k+1}'}\log_{(q_{2k+1}')^{e_{2k+2}-2}}q_{2k+1}'^2\\
&=\frac{2}{e^+-2}.
\end{align*}
If we pick $i=c_k'+[\frac{c_{2k+2}}{2}]$, $i'=[\frac{c_{2k+2}}{2}]$, we have $i'>\frac{c_{2k+2}}{3}$ and
$$\frac{1}{(a_{2k}'-i'a_{2k+1}')a_{2k}'}<\frac{2}{a_{2k}'^2},$$
therefore,
\begin{align*}
\liminf\limits_{i\rightarrow+\infty, i\neq c_k'}\log_{i'}\frac{1}{(a_{2k}'-i'a_{2k+1}')a_{2k}'}&\leq\liminf\limits_{i\rightarrow+\infty\atop c_k'<i=c_k'+\left[\frac{c_{2k+2}}{2}\right]<c_{k+1}'}\log_{i'}\frac{1}{(a_{2k}'-i'a_{2k+1}')a_{2k}'}\\
&\leq\liminf\limits_{i\rightarrow+\infty\atop c_k'<i=c_k'+[\frac{c_{2k+2}}{2}]<c_{k+1}'}\log_{\frac{c_{2k+2}}{3}}\frac{2}{a_{2k}'^2}\\
&=\liminf\limits_{i\rightarrow+\infty\atop c_k'<i=c_k'+\left[\frac{c_{2k+2}}{2}\right]<c_{k+1}'}\log_{(q_{2k+1}')^{e_{2k+2}-2}}q_{2k+1}'^2\\
&=\frac{2}{e^+-2}.
\end{align*}

Thus, we can obtain
$$\liminf\limits_{i\rightarrow+\infty, i\neq c_k'}\log_{i'}\frac{1}{(a_{2k}'-i'a_{2k+1}')a_{2k}'}=\frac{2}{e^+-2}.$$
Then by (\ref{687653}), we have
\begin{equation}
\limsup\limits_{i\rightarrow+\infty, i\neq c_k'}\log_{n_i}i'=\frac{1}{\frac{2}{e^+-2}+1}=\frac{e^+-2}{e^+}. \label{737676}
\end{equation}

Combining (\ref{737674}), (\ref{737675}), and (\ref{737676}), we can conclude that
\begin{equation}\limsup\limits_{i\rightarrow+\infty, i\neq c_k'}\log_{n_i}|f_{n_i}|\leq\max\left\{\frac{1}{2},\frac{e^+-2}{e^+}\right\}.\label{7654}\end{equation}

If we pick $i=c_k'+1$, $i'=1$, by (\ref{666}), we have $f_{n_i}>\frac{r_3}{a_{n_i}}-1$, which means that
$$\limsup\limits_{i\rightarrow+\infty, i\neq c_k'}\log_{n_i}|f_{n_i}|\geq\limsup\limits_{i\rightarrow+\infty, i=c_k'+1<c_{k+1}'}\log_{n_i}\left(\frac{r_3}{a_{n_i}}-1\right),$$
\begin{equation}=\limsup\limits_{i\rightarrow+\infty, i=c_k'+1<c_{k+1}'}\log_{n_i}\frac{1}{a_{n_i}}\geq\frac{1}{2}.\label{7655}\end{equation}

When $\frac{e^+-2}{e^+}\leq\frac{1}{2}$, by (\ref{7654}) and (\ref{7655}), we can deduce that
\begin{equation}
\limsup\limits_{i\rightarrow+\infty, i\neq c_k'}\log_{n_i}|f_{n_i}|=\frac{1}{2}.  \label{387462}
\end{equation}

When $\frac{e^+-2}{e^+}>\frac{1}{2}$, $e^+>4$, there exists an increasing even number sequence $\{s_j\}_{j=1}^{+\infty}$, such that $e_{s_j}\rightarrow 4$ and $e_{s_j}>3$. If we pick $i=c_{\frac{s_j}{2}}'+\left[\frac{c_{s_j+2}}{2}\right]$, $i'=\left[\frac{c_{s_j+2}}{2}\right]$, then
$$\frac{1}{a_{n_i}}<\frac{2}{a_{s_j}}<\frac{4q_{s_j+1}'}{a}=o\left({q_{s_j+1}'}^{e_{s_j+2}-2}\right)=o(c_{s_j+2}),$$
thus
\begin{align*}
\limsup\limits_{i\rightarrow+\infty, i\neq c_k'}\log_{n_i}|f_{n_i}|&\geq\limsup\limits_{i\rightarrow+\infty\atop c_{\frac{s_j}{2}}'<i=c_{\frac{s_j}{2}}'+\left[\frac{c_{s_j+2}}{2}\right]<c_{{\frac{s_j}{2}}+1}'}\log_{n_i}|f_{n_i}|\\
&=\limsup\limits_{i\rightarrow+\infty\atop c_{\frac{s_j}{2}}'<i=c_{\frac{s_j}{2}}'+\left[\frac{c_{s_j+2}}{2}\right]<c_{{\frac{s_j}{2}}+1}'}\log_{n_i}i'\\
&=\frac{1}{\liminf\limits_{i\rightarrow+\infty\atop c_{\frac{s_j}{2}}'<i=c_{\frac{s_j}{2}}'+\left[\frac{c_{s_j+2}}{2}\right]<c_{{\frac{s_j}{2}}+1}'}\log_{i'}\frac{1}{(a_{s_j}'-i'a_{s_j+1}')a_{s_j}'}+1}\\
&\geq\frac{1}{\frac{2}{e^+-2}+1}\\
&=\frac{e^+-2}{e^+}.
\end{align*}
Therefore,
\begin{equation}
\limsup\limits_{i\rightarrow+\infty, i\neq c_k'}\log_{n_i}|f_{n_i}|=\frac{e^+-2}{e^+}.  \label{387463}
\end{equation}

In summary, by (\ref{387462}) and (\ref{387463}), we have
\begin{equation}
\limsup\limits_{i\rightarrow+\infty, i\neq c_k'}\log_{n_i}|f_{n_i}|=\max\left\{\frac{1}{2},\frac{e^+-2}{e^+}\right\}.    \label{665}
\end{equation}

Combining (\ref{664}) and (\ref{665}), we can conclude that
\begin{align}
\limsup\limits_{i\rightarrow+\infty}\log_{n_i}f_{n_i}&=\max\left\{\limsup_{k\rightarrow+\infty}\log_{n_{c_k'}}f_{n_{c_k'}},\limsup\limits_{i\rightarrow+\infty, i\neq c_k'}\log_{n_i}f_{n_i}\right\}\nonumber\\   \nonumber
&=\max\left\{\frac{e^--1}{e^-},\max\left\{\frac{1}{2},\frac{e^+-2}{e^+}\right\}\right\}\\
&=\max\left\{\frac{e^--1}{e^-},\frac{e^+-2}{e^+}\right\}.       \label{6612}
\end{align}

Then, in order to prove the main result, we only need to prove
$$\limsup_{n\rightarrow+\infty}\log_nf_n=\limsup\limits_{i\rightarrow+\infty}\log_{n_i}f_{n_i}.$$\\

\subsection{Overall monotonicity of $k$}

\begin{ttt}   \label{773}
For an excursion $(a_i,b_i)_{i=1}^m$ where $m\geq4$, we have
$$f_1>f_i>f_{m-1}$$
for $i\in[2,m-2]$.

\end{ttt}

\begin{proof}

For any $n\in[2,m-1]$,
$$f_n-f_1=h(a_0,a_1,\cdots,a_{n+1})-(k_1-3).$$
Then, following a similar approach as in \S3.2.1, if we eliminate a largest term from the sequence $\{a_0,a_1,a_2,\cdots,a_{n+1}\}$, by Lemma \ref{97743}, we know that the value of $h$ remains unchanged. We continue this process until we cannot continue further. Let the final reduced sequence be $$\{a_0,a_1,a_{m_1},a_{m_2},\cdots,a_{m_{k-1}},a_n,a_{n+1}\}$$
where $k\geq1$, we define $m_0=1$, $m_k=n$.

We have
$$h(a_0,a_1,\cdots,a_{n+1})=h(a_0,a_1,a_{m_1},a_{m_2},\cdots,a_{m_{k-1}},a_n,a_{n+1}).$$
Since for any $j\in(1,m)$, $a_1<a_j$, so we know $a_1<a_{m_1}$, which implies that $a_{m_1}\leq a_{m_2}$, otherwise, the process of elimination wouldn't have stopped yet. Considering that throughout this process, every term divides the sum of two neighboring terms, we have $a_{m_1}<a_{m_2}$, otherwise, $\frac{a_1+a_{m_2}}{a_{m_1}}\in(1,2)$. Then we can deduce that $a_{m_2}<a_{m_3}$ and so on, which means that
$$a_{m_0}=a_1<a_{m_1}<a_{m_2}<\cdots<a_{m_k}=a_n.$$
Let $s_j=\frac{a_{m_j}}{a_{m_{j-1}}}>1$ for $j\in[1,k]$. Then, since $a_{m_{j-1}}<a_{m_j}<a_{m_{j+1}}$ for $j\in[1,k-1]$, we have $\frac{a_{m_{j-1}}+a_{m_{j+1}}}{a_{m_j}}=\left\lceil \frac{a_{m_{j+1}}}{a_{m_j}}\right\rceil=\lceil s_{j+1}\rceil$. Therefore,
\begin{align*}
f_n-f_1=&h(a_0,a_1,\cdots,a_{n+1})-(k_1-3)\\
=&h(a_0,a_1,a_{m_1},a_{m_2},\cdots,a_{m_{k-1}},a_n,a_{n+1})-(k_1-3)\\
=&\left(\frac{a_0+a_{m_1}}{a_1}-3\right)+\sum_{j=1}^{k-1}\left(\frac{a_{m_{j-1}}+a_{m_{j+1}}}{a_{m_j}}-3\right)\\
&+\left(\frac{a_{m_{k-1}}+a_{n+1}}{a_n}-3\right)-\left(\frac{a_0+a_2}{a_1}-3\right)\\
=&s_1-3+\sum_{j=2}^k(\lceil s_j\rceil-3)+\left\lceil\frac{a_{n+1}}{a_n}\right\rceil-\frac{a_2}{a_1}.
\end{align*}

Firstly, we will prove $f_n<f_1$ for $n\in[2,m-1]$.
\begin{align*}
f_n-f_1=&s_1-3+\sum_{j=2}^k(\lceil s_j\rceil-3)+\left\lceil\frac{a_{n+1}}{a_n}\right\rceil-\frac{a_2}{a_1}\\
<&s_1-3+\sum_{j=2}^k(s_j+1-3)+\left\lceil\frac{1}{a_n}\right\rceil-\frac{1-a_1}{a_1}\\
<&s_1-3+\sum_{j=2}^k(s_j-2)+\left(\frac{1}{a_n}+1\right)-\frac{1}{a_1}+1\\
=&-\frac{\prod\limits_{j=1}^ks_j}{a_n}+\frac{1}{a_n}+\sum_{j=1}^ks_j-(2k-1).
\end{align*}
Let $s_0=\frac{1}{a_n}\geq1$, then
$$f_n-f_1<-\prod\limits_{j=0}^ks_j+\sum_{j=0}^ks_j-(2k-1).$$

When $k=1$,
$$f_n-f_1<-s_0s_1+s_0+s_1-1=-(s_0-1)(s_1-1)\leq0;$$
when $k>1$,
\begin{align*}
f_n-f_1<&-\prod\limits_{j=0}^ks_j+\sum_{j=0}^ks_j-(2k-1)\\
=&-\prod\limits_{j=0}^{k-1}s_j+\sum_{j=0}^{k-1}s_j-(2k-3)-(s_k-1)\prod\limits_{j=0}^{k-1}s_j+s_k-2\\
<&-\prod\limits_{j=0}^{k-1}s_j+\sum_{j=0}^{k-1}s_j-(2k-3)-(s_k-1)+s_k-2\\
<&-\prod\limits_{j=0}^{k-1}s_j+\sum_{j=0}^{k-1}s_j-(2k-3)\\
\vdots&\\
<&-s_0s_1+s_0+s_1-1\\
\leq&0.
\end{align*}

Since we have prove for any $n\in[2,m-2]$,
$$f_n-f_1=h(a_1,\cdots,a_{n+1})<0,$$
and the only condition required is that $a_i>a_1$ for $i\in[2,n+1]$. Therefore by symmetry, as $a_i>a_m$ for $i\in[2,m-1]$, we know for any $n\in[2,m-2]$,
$$f_{m-1}-f_n=h(a_n,a_{n+1},\cdots,a_m)=h(a_m,a_{m-1},\cdots,a_n)<0.$$\\

\end{proof}

\subsection{$n\in(n_{i-1},n_i),i\in[c_k'+1,c_{k+1}'-1]$}

For $j\in(n_{i-1},n_i)$ and $i\in[c_k'+1,c_{k+1}'-1]$, since it is an excursion from $(a_{n_{i-1}},b_{n_{i-1}})$ to $(a_{n_i},b_{n_i})$, by Theorem \ref{773}, we have
$$f_j\in[f_{n_i-1},f_{n_{i-1}}).$$

Using (\ref{666}), we know that as $i,k\rightarrow+\infty$, $f_{n_i}>-i'$. Additionally, $k_{n_i}=\frac{a_{n_i-1}+a_{n_i+1}}{a_{n_i}}\leq\frac{2}{a_{n_i}}$. Consequently,
$$f_{n_i-1}=f_{n_i}-k_{n_i}>-i'-\frac{2}{a_{n_i}}\geq-2\max\left\{\frac{1}{a_{n_i}},i'\right\}.$$
Since $a_{n_i}=a_{2k}'-i'a_{2k+1}'>a_{2k+1}'$ and $a_{n_{i-1}}=a_{2k}'-(i'-1)a_{2k+1}'<2(a_{2k}'-i'a_{2k+1}')=2a_{n_i}$, we have
$$\max\left\{\frac{1}{a_{n_i}},i'\right\}<2\max\left\{\frac{1}{a_{n_{i-1}}},i'-1\right\},$$
$$f_j\geq f_{n_i-1}>-2\max\left\{\frac{1}{a_{n_i}},i'\right\}>-4\max\left\{\frac{1}{a_{n_{i-1}}},i'-1\right\}.$$
Therefore,
$$\log_j|f_j|<\log_{n_{i-1}}\max\left\{|f_{n_{i-1}}|,\frac{4}{a_{n_{i-1}}},4(i'-1)\right\}.$$
Since $|f_{n_i}|<\max\left\{\frac{r_4}{a_{n_i}},i'\right\}$ for $i\neq c_k'$, and $\left|f_{n_{c_k'}}\right|<\frac{r_2}{a_{n_{c_k'}}}$,  we obtain
$$\limsup\limits_{\substack{j\rightarrow+\infty,j\in(n_{i-1},n_i)\\ i\in[c_k'+1,c_{k+1}'-1]}}\log_j|f_j|\leq
\limsup\limits_{\substack{i\rightarrow+\infty\\ i\in[c_k'+1,c_{k+1}'-1]}}\log_{n_{i-1}}\max\left\{|f_{n_{i-1}}|,\frac{4}{a_{n_{i-1}}},4(i'-1)\right\}$$
\begin{align}
\leq&\max\left\{\limsup_{k\rightarrow+\infty}\log_{n_{c_k'}}\frac{1}{a_{n_{c_k'}}},\max\left\{\limsup\limits_{i\rightarrow+\infty, i\neq c_k'}\log_{n_i}\frac{1}{a_{n_i}},\limsup\limits_{i\rightarrow+\infty, i\neq c_k'}\log_{n_i}i'\right\}\right\}\nonumber\\
=&\max\left\{\frac{e^--1}{e^-},\frac{e^+-2}{e^+}\right\}.    \label{6610}
\end{align}\\

\subsection{$n\in(n_{c_k'-1},n_{c_k'})$}

For $j\in(n_{c_k'-1},n_{c_k'})$, according to Remark \ref{64}, we have
$$\left\{a_m|m\in\left[n_{c_k'-1}+1,n_{c_k'}-1\right]\right\}$$$$=\left\{sa_{n_{c_k'-1}}+ta_{n_{c_k'}}|s,t\in\mathbb{Z}^+, (s,t)=1, sa_{n_{c_k'-1}}+ta_{n_{c_k'}}\leq1\right\}.$$
Let $a_j=s_ja_{n_{c_k'-1}}+t_ja_{n_{c_k'}}$ for $j\in\left[n_{c_k'-1}+1,n_{c_k'}-1\right]$, where we have $\left(s_{n_{c_k'-1}},t_{n_{c_k'-1}}\right)=(1,0)$ and $\left(s_{n_{c_k'}},t_{n_{c_k'}}\right)=(0,1)$.
Since $a_{n_{c_k'-1}}+c_{2k+1}a_{n_{c_k'}}=a_{2k}'+a_{2k-1}'+c_{2k+1}a_{2k}'\leq 3a_{2k-1}'<1$, we define $m_i$ as the number such that $(s_{m_i},t_{m_i})=(1,i)$ for $i\in[0,c_{2k+1}]$. Notably, $m_0=n_{c_k'-1}$.

For all $j\in\left(m_{c_{2k+1}},n_{c_k'}\right)$, we have $\frac{t_j}{s_j}>\frac{c_{2k+1}}{1}$, which implies $t_j>c_{2k+1}$. Therefore, $$a_j=s_ja_{n_{c_k'-1}}+t_ja_{n_{c_k'}}>a_{n_{c_k'-1}}+c_{2k+1}a_{n_{c_k'}}=a_{m_{c_{2k+1}}}.$$
Moreover, $a_j>a_{n_{c_k'}}$. These indicate that it is an excursion from $(a_{m_{c_{2k+1}}},b_{m_{c_{2k+1}}})$ to $(a_{n_{c_k'}},b_{n_{c_k'}})$. By Theorem \ref{773}, we obtain
\begin{equation}f_{n_{c_k'}-1}\leq f_j< f_{m_{c_{2k+1}}},\label{37373}\end{equation}
$\forall j\in(m_i,m_{i+1})$ for $i\in[0,c_{2k+1}-1]$. We also know that $\frac{t_j}{s_j}\in\left(\frac{i}{1},\frac{i+1}{1}\right)$ so $t_j\geq i+1$, $s_j>1$, which means that  $$a_j=s_ja_{n_{c_k'-1}}+t_ja_{n_{c_k'}}>a_{n_{c_k'-1}}+(i+1)a_{n_{c_k'}}=a_{m_{i+1}}>a_{m_i}.$$
Thus, it is an excursion from $(a_{m_i},b_{m_i})$ to $(a_{m_{i+1}},b_{m_{i+1}})$, by Theorem \ref{773}, we have
\begin{equation}f_{m_{i+1}-1}\leq f_j< f_{m_i}.\label{73737}\end{equation}\\

\subsubsection{$n=m_i$}

Now, we aim to show that $\limsup\limits_{j\rightarrow+\infty\atop j=m_i\in\left(n_{c_k'-1},n_{c_k'}\right)}\log_{m_i}|f_{m_i}|\leq\max\left\{\frac{e^--1}{e^-},\frac{e^+-2}{e^+}\right\}$. Firstly, we provide an estimation of $f_{m_i}=h(a_0,a_1,\cdots,a_{m_i+1})$ for $i\in[1,c_{2k+1}]$.
\begin{align*}
f_{m_i}&=h(a_{-1},a_0,a_1,\cdots,a_{m_i+1})-(k_0-3)\\
&=h\left(a_{-1},a_0,a_{n_0},a_{n_1},a_{n_2},\cdots,a_{n_{c_k'-1}},a_{m_1},a_{m_2},\cdots,a_{m_i},a_{m_i+1}\right)-(k_0-3)\\
&=\sum_{j=0}^{k-1}(c_{2j+1}-c_{2j})-c_{2k}-(i-1)+\left(\frac{a_{m_{i-1}}+a_{m_i+1}}{a_{m_i}}-3\right).
\end{align*}
We have
$$\frac{a_{m_{i-1}}}{a_{m_i}}<1,$$
$$\frac{1}{a_{m_i}}-1=\frac{1-a_{m_i}}{a_{m_i}}<\frac{a_{m_i+1}}{a_{m_i}}\leq\frac{1}{a_{m_i}},$$
which implies
$$\sum_{j=0}^{k-1}(c_{2j+1}-c_{2j})-c_{2k}-i+\frac{1}{a_{m_i}}-3<f_{m_i}<\sum_{j=0}^{k-1}(c_{2j+1}-c_{2j})-c_{2k}-i+\frac{1}{a_{m_i}}-1.$$
Since $a_{m_i}=a_{2k}'+a_{2k-1}'+ia_{2k}'\leq 3a_{2k-1}'$, by Remark \ref{771}, we have
$$\left|\frac{\sum\limits_{j=0}^k(c_{2j}-c_{2j-1})}{1\big/ a_{m_i}}\right|\leq\left|3\cdot\frac{\sum\limits_{j=0}^k(c_{2j}-c_{2j-1})}{1\big/ a_{2k-1}'}\right|\rightarrow0.$$
Therefore, for constant $r_5\in(0,1)$ and $r_6\in(1,+\infty)$, as $k\rightarrow+\infty$, we obtain
\begin{equation}
\frac{r_5}{a_{m_i}}-i<f_{m_i}<\frac{r_6}{a_{m_i}}-i.       \label{667}
\end{equation}
Let $r_5=0.9$ and $r_6=1.1$. Then we have
$$|f_{m_i}|<\max\left\{\frac{r_6}{a_{n_i}},i\right\},$$
therefore,
\begin{align*}
\limsup\limits_{j\rightarrow+\infty\atop j=m_i\in\left(n_{c_k'-1},n_{c_k'}\right)}\log_{m_i}|f_{m_i}|&\leq\limsup\limits_{j\rightarrow+\infty\atop j=m_i\in\left(n_{c_k'-1},n_{c_k'}\right)}\log_{m_i}\max\left\{\frac{r_6}{a_{m_i}},i\right\}\\
&=\max\left\{\limsup\limits_{j\rightarrow+\infty\atop j=m_i\in\left(n_{c_k'-1},n_{c_k'}\right)}\log_{m_i}\frac{r_6}{a_{m_i}},\limsup\limits_{j\rightarrow+\infty\atop j=m_i\in\left(n_{c_k'-1},n_{c_k'}\right)}\log_{m_i}i\right\}\\
&=\max\left\{\limsup\limits_{j\rightarrow+\infty\atop j=m_i\in\left(n_{c_k'-1},n_{c_k'}\right)}\log_{m_i}\frac{1}{a_{m_i}},\limsup\limits_{j\rightarrow+\infty\atop j=m_i\in\left(n_{c_k'-1},n_{c_k'}\right)}\log_{m_i}i\right\}.
\end{align*}
Since $m_i>n_{c_k'-1}$, we have $\frac{1}{a_{m_i}}<\frac{1}{a_{n_{c_k'-1}}}$. Therefore,
$$\limsup\limits_{j\rightarrow+\infty\atop j=m_i\in\left(n_{c_k'-1},n_{c_k'}\right)}\log_{m_i}\frac{1}{a_{m_i}}\leq\limsup\limits_{i\rightarrow+\infty, i= c_k'-1}\log_{n_i}\frac{1}{a_{n_i}}\leq\limsup\limits_{i\rightarrow+\infty, i\neq c_k'}\log_{n_i}\frac{1}{a_{n_i}}=\frac{1}{2}.$$\\
As $m_i\in(n_{c_k'-1},n_{c_k'})$ and it is an excursion from $(a_{m_j},b_{m_j})$ to $(a_{m_{j+1}},b_{m_{j+1}})$, we can express $m_i$ as follows:
\begin{align*}
m_i&=m_i-n_{c_k'-1}+n_{c_k'-1}\\
&=\sum_{j=1}^i(m_j-m_{j-1})+n_{c_k'-1}\\
&\sim\frac{3}{\pi^2}\sum_{j=1}^{i}\frac{1}{a_{m_j}a_{m_{j-1}}}+n_{c_k'-1}\\
&=\frac{3}{\pi^2}\sum_{j=1}^{i}\frac{1}{(a_{2k-1}'+ja_{2k}')(a_{2k-1}'+(j+1)a_{2k}')}+n_{c_k'-1}\\
&=\frac{3}{\pi^2}\sum_{j=1}^{i}\frac{1}{a_{2k}'}\left(\frac{1}{a_{2k-1}'+ja_{2k}'}-\frac{1}{a_{2k-1}'+(j+1)a_{2k}'}\right)+n_{c_k'-1}\\
&=\frac{3}{\pi^2}\frac{1}{a_{2k}'}\left(\frac{1}{a_{2k-1}'+a_{2k}'}-\frac{1}{a_{2k-1}'+(i+1)a_{2k}'}\right)+n_{c_k'-1}\\
&=\frac{3}{\pi^2}\frac{i}{(a_{2k-1}'+a_{2k}')(a_{2k-1}'+(i+1)a_{2k}')}+n_{c_k'-1}.
\end{align*}\\
Given that $a_{2k-1}'+(i+1)a_{2k}'<3a_{2k-1}'$, we have
$$\frac{i}{(a_{2k-1}'+a_{2k}')(a_{2k-1}'+(i+1)a_{2k}')}\in\left(\frac{i}{6a_{2k-1}'^2},\frac{i}{a_{2k-1}'^2}\right).$$
By Lemma \ref{9932}, let $s=c_{2k}-1$, we obtain
\begin{equation}
n_{c_k'-1}=O\left((\frac{1}{a_{2k-1}'})^2\right). \label{9939}
\end{equation}
Furthermore,
\begin{equation}
m_i\asymp\frac{i}{a_{2k-1}'^2}.  \label{669}
\end{equation}
As a result,
$$\limsup\limits_{j\rightarrow+\infty\atop j=m_i\in\left(n_{c_k'-1},n_{c_k'}\right)}\log_{m_i}i=\limsup\limits_{j\rightarrow+\infty\atop j=m_i\in\left(n_{c_k'-1},n_{c_k'}\right)}\log_{\frac{i}{a_{2k-1}'^2}}i.$$
Since $\frac{1}{a_{2k-1}'^2}>1$,  $\log_{\frac{x}{a_{2k-1}'^2}}x$ is a monotonically increasing function for $x\geq1$. Therefore,
$$\log_{\frac{i}{a_{2k-1}'^2}}i\leq\log_{\frac{c_{2k+1}}{a_{2k-1}'^2}}c_{2k+1}.$$
Hence,
\begin{align*}
\limsup\limits_{j\rightarrow+\infty\atop j=m_i\in\left(n_{c_k'-1},n_{c_k'}\right)}\log_{m_i}i&\leq\limsup\limits_{k\rightarrow+\infty}\log_{\frac{c_{2k+1}}{a_{2k-1}'^2}}c_{2k+1}\\
&=\limsup\limits_{k\rightarrow+\infty}\log_{c_{2k+1}q_{2k}'^2}c_{2k+1}\\
&=\limsup\limits_{k\rightarrow+\infty}\log_{{q_{2k}'}^{e_{2k+1}-2}q_{2k}'^2}{q_{2k}'}^{e_{2k+1}-2}\\
&=\frac{e^--2}{e^-}.
\end{align*}

Therefore,
\begin{align}
\limsup\limits_{j\rightarrow+\infty\atop j=m_i\in\left(n_{c_k'-1},n_{c_k'}\right)}\log_{m_i}|f_{m_i}|&\leq\max\left\{\limsup\limits_{j\rightarrow+\infty\atop j=m_i\in\left(n_{c_k'-1},n_{c_k'}\right)}\log_{m_i}\frac{1}{a_{m_i}},\limsup\limits_{j\rightarrow+\infty\atop j=m_i\in\left(n_{c_k'-1},n_{c_k'}\right)}\log_{m_i}i\right\}\nonumber\\ \nonumber
&\leq\max\left\{\frac{1}{2},\frac{e^--2}{e^-}\right\}\\
&\leq\max\left\{\frac{e^--1}{e^-},\frac{e^+-2}{e^+}\right\}.     \label{668}
\end{align}\\

\subsubsection{$n\in(m_{i-1},m_i),i\in[1,c_{2k+1}]$}

For $j\in(m_{i-1},m_i)$ with $i\in[1,c_{2k+1}]$, by (\ref{73737}), we have
$$f_{m_{i}-1}\leq f_j< f_{m_{i-1}}.$$
By (\ref{667}), we know that as $i,k\rightarrow+\infty$, $f_{m_i}>-i$. Moreover, $k_{m_i}=\frac{a_{m_i-1}+a_{m_i+1}}{a_{m_i}}\leq\frac{2}{a_{m_i}}$. So we have
$$f_{m_i-1}=f_{m_i}-k_{m_i}>-i-\frac{2}{a_{m_i}}\geq-2\max\left\{\frac{1}{a_{m_i}},i\right\}.$$
Since $a_{m_i}>a_{m_{i-1}}$, we have
$$\max\left\{\frac{1}{a_{m_i}},i\right\}<2\max\left\{\frac{1}{a_{m_{i-1}}},i-1\right\},$$
$$f_j\geq f_{m_i-1}>-2\max\left\{\frac{1}{a_{m_i}},i\right\}>-4\max\left\{\frac{1}{a_{m_{i-1}}},i-1\right\},$$
then
$$\log_j|f_j|<\log_{m_{i-1}}\max\left\{|f_{m_{i-1}}|,\frac{4}{a_{m_{i-1}}},4(i-1)\right\}.$$
Since $|f_{m_i}|<\max\left\{\frac{r_6}{a_{n_i}},i\right\}$ for $i\in[1,c_{2k+1}]$, and $|f_{m_0}|=\left|f_{n_{c_k'-1}}\right|$,  we have
$$\limsup\limits_{\substack{j\rightarrow+\infty,j\in(m_{i-1},m_i)\\ m_i\in\left(n_{c_k'-1},n_{c_k'}\right)}}\log_j|f_j|\leq
\limsup\limits_{j\rightarrow+\infty\atop j=m_{i-1}\in\left[n_{c_k'-1},n_{c_k'}\right)}\log_{m_{i-1}}\max\left\{|f_{m_{i-1}}|,\frac{4}{a_{m_{i-1}}},4(i-1)\right\},$$
\begin{align*}
\leq&\max\left\{\limsup_{j\rightarrow+\infty\atop j=n_{c_k'-1}}\log_j|f_j|,\max\left\{\limsup\limits_{j\rightarrow+\infty\atop j=m_i\in\left(n_{c_k'-1},n_{c_k'}\right)}\log_{m_i}\frac{1}{a_{m_i}},\limsup\limits_{j\rightarrow+\infty\atop j=m_i\in\left(n_{c_k'-1},n_{c_k'}\right)}\log_{m_i}i\right\}\right\}\\
\leq&\max\left\{\limsup_{i\rightarrow+\infty\atop i\neq c_k'}\log_{n_i}|f_{n_i}|,\max\left\{\frac{1}{2},\frac{e^--2}{e^-}\right\}\right\}\\
\leq&\max\left\{\frac{e^--1}{e^-},\frac{e^+-2}{e^+}\right\}.
\end{align*}\\

\subsubsection{$n\in\left(m_{c_{2k+1}},n_{c_k'}\right)$}

For $j\in\left(m_{c_{2k+1}},n_{c_k'}\right)$, by (\ref{37373}), we have
$$f_{n_{c_k'}-1}\leq f_j< f_{m_{c_{2k+1}}}.$$
By (\ref{661}), we know that as $k\rightarrow+\infty$, $f_{n_{c_k'}}>\frac{r_1}{a_{n_{c_k'}}}$. Also, $k_{n_{c_k'}}=\frac{a_{n_{c_k'}-1}+a_{n_{c_k'}+1}}{a_{n_{c_k'}}}\leq\frac{2}{a_{n_{c_k'}}}$. So we have
$$f_j\geq f_{n_{c_k'}-1}=f_{n_{c_k'}}-k_{n_{c_k'}}>\frac{r_1}{a_{n_{c_k'}}}-\frac{2}{a_{n_{c_k'}}}>-\frac{1}{a_{n_{c_k'}}}>-\frac{f_{n_{c_k'}}}{r_1},$$
thus
\begin{align*}
\limsup_{j\rightarrow+\infty\atop j\in\left(m_{c_{2k+1}},n_{c_k'}\right)}\log_j|f_j|\leq\max\Bigg\{&\limsup_{j\rightarrow+\infty\atop j=m_{c_{2k+1}}\in\left(n_{c_k'-1},n_{c_k'}\right)}\log_{m_{c_{2k+1}}}|f_{m_{c_{2k+1}}}|,\\
&\limsup_{j\rightarrow+\infty\atop j=m_{c_{2k+1}}\in\left(n_{c_k'-1},n_{c_k'}\right)}\log_{m_{c_{2k+1}}}|\frac{f_{n_{c_k'}}}{r_1}|\Bigg\}.
\end{align*}
By (\ref{668}), we have
\begin{align*}
\limsup_{j\rightarrow+\infty\atop j=m_{c_{2k+1}}\in\left(n_{c_k'-1},n_{c_k'}\right)}\log_{m_{c_{2k+1}}}|f_{m_{c_{2k+1}}}|&\leq\limsup\limits_{j\rightarrow+\infty\atop j=m_i\in\left(n_{c_k'-1},n_{c_k'}\right)}\log_{m_i}|f_{m_i}|\\
&\leq\max\left\{\frac{e^--1}{e^-},\frac{e^+-2}{e^+}\right\}.
\end{align*}
Since $c_{2k+1}a_{2k}'<a_{2k-1}'=a_{2k+1}'+c_{2k+1}a_{2k}'<2c_{2k+1}a_{2k}'$, then by (\ref{669}), we have
$$m_{c_{2k+1}}\asymp\frac{c_{2k+1}}{a_{2k-1}'^2}\asymp\frac{1}{a_{2k-1}'a_{2k}'}\asymp n_{c_k'},$$
therefore
\begin{align*}
\limsup_{j\rightarrow+\infty\atop j=m_{c_{2k+1}}\in\left(n_{c_k'-1},n_{c_k'}\right)}\log_{m_{c_{2k+1}}}\left|\frac{f_{n_{c_k'}}}{r_1}\right|&=\limsup_{j\rightarrow+\infty\atop j=m_{c_{2k+1}}\in\left(n_{c_k'-1},n_{c_k'}\right)}\log_{n_{c_k'}}\left|\frac{f_{n_{c_k'}}}{r_1}\right|\\
&=\limsup_{k\rightarrow+\infty}\log_{n_{c_k'}}|f_{n_{c_k'}}|\\
&\leq\max\left\{\frac{e^--1}{e^-},\frac{e^+-2}{e^+}\right\},
\end{align*}
so
$$\limsup_{j\rightarrow+\infty\atop j\in\left(m_{c_{2k+1}},n_{c_k'}\right)}\log_j|f_j|\leq\max\left\{\frac{e^--1}{e^-},\frac{e^+-2}{e^+}\right\}.$$

In conclusion,
\begin{align}
\limsup_{j\rightarrow+\infty\atop j\in\left(n_{c_k'-1},n_{c_k'}\right)}\log_j|f_j|=&\max\bigg\{\limsup\limits_{j\rightarrow+\infty\atop j=m_i\in\left(n_{c_k'-1},n_{c_k'}\right)}\log_{m_i}|f_{m_i}|,\nonumber\\  \nonumber
&\limsup\limits_{\substack{j\rightarrow+\infty,j\in(m_{i-1},m_i)\\  m_i\in(n_{c_k'-1},n_{c_k'})}}\log_j|f_j|,\limsup_{j\rightarrow+\infty\atop j\in\left(m_{c_{2k+1}},n_{c_k'}\right)}\log_j|f_j|\bigg\}\\
\leq&\max\left\{\frac{e^--1}{e^-},\frac{e^+-2}{e^+}\right\}.    \label{6611}
\end{align}\\

\subsection{Summary}

Thus, by (\ref{6610})(\ref{6611}), we have
\begin{align}
\limsup\limits_{j\rightarrow+\infty,j\neq n_i}\log_j|f_j|&\leq\max\left\{\limsup\limits_{\substack{j\rightarrow+\infty,j\in(n_{i-1},n_i)\\  \nonumber i\in[c_k'+1,c_{k+1}'-1]}}\log_j|f_j|,\limsup_{j\rightarrow+\infty\atop j\in\left(n_{c_k'-1},n_{c_k'}\right)}\log_j|f_j|\right\}\\
&\leq\max\left\{\frac{e^--1}{e^-},\frac{e^+-2}{e^+}\right\},    \label{6613}
\end{align}
then by (\ref{6612})(\ref{6613})
$$\limsup_{n\rightarrow+\infty}\log_n|f_n|=\limsup\limits_{i\rightarrow+\infty}\log_{n_i}|f_{n_i}|=\max\left\{\frac{e^--1}{e^-},\frac{e^+-2}{e^+}\right\},$$
which means that we have already proved the Theorem \ref{8817}.

By symmetry, we can get Theorem \ref{8818}.\\

\subsection{Proof of Corollary \ref{9898}}
\begin{proof}
We only need to demonstrate that Theorem \ref{8817} remains valid if we replace $k_i$ by $\hat{k}\circ T^{i-1}-3$. Let $\hat{f}_n:=\displaystyle\sum_{i=1}^n\hat{k}\left(T^{i-1}(a,b)-3\right)=\displaystyle\sum_{i=1}^n\left(\frac{k_{i-1}+k_i}{2}-3\right).$

Since
$$\hat{f}_n=f_{n-1}+\frac{k_n}{2}+\frac{k_0}{2}-3\geq f_{n-1}+\frac{1}{2}+\frac{k_0}{2}-3,$$
we have
\begin{equation}
\limsup_{n\rightarrow+\infty}\frac{\log \hat{f}_n}{\log n}\geq\limsup_{n\rightarrow+\infty}\frac{\log f_n}{\log n}, \label{8819}
\end{equation}
\begin{equation}
\liminf_{n\rightarrow+\infty}\frac{\log \hat{f}_n}{\log n}\geq\liminf_{n\rightarrow+\infty}\frac{\log f_n}{\log n}. \label{8822}
\end{equation}
On the other hand, since
$$f_n=\hat{f}_n+\frac{k_n}{2}-\frac{k_0}{2}\geq \hat{f}_n+\frac{1}{2}-\frac{k_0}{2},$$
we also have
\begin{equation}
\limsup_{n\rightarrow+\infty}\frac{\log f_n}{\log n}\geq\limsup_{n\rightarrow+\infty}\frac{\log \hat{f}_n}{\log n},\label{8820}
\end{equation}
\begin{equation}
\liminf_{n\rightarrow+\infty}\frac{\log f_n}{\log n}\geq\liminf_{n\rightarrow+\infty}\frac{\log \hat{f}_n}{\log n}.  \label{8821}
\end{equation}

By combining (\ref{8819}) and (\ref{8820}), we conclude that
$$\limsup_{n\rightarrow+\infty}\frac{\log \hat{f}_n}{\log n}=\limsup_{n\rightarrow+\infty}\frac{\log f_n}{\log n}.$$
Similarly, by combining (\ref{8822}) and (\ref{8821}), we find that
$$\liminf_{n\rightarrow+\infty}\frac{\log \hat{f}_n}{\log n}=\liminf_{n\rightarrow+\infty}\frac{\log f_n}{\log n}.$$
Therefore, we have
$$\limsup_{n\rightarrow+\infty}\frac{\log |\hat{f}_n|}{\log n}=\limsup_{n\rightarrow+\infty}\frac{\log |f_n|}{\log n}=\max\left\{\frac{e^--1}{e^-},\frac{e^+-2}{e^+}\right\}.$$

\end{proof}

\end{document}